\documentclass[11pt]{amsart}  

\usepackage{amsmath}
\usepackage{amsthm}
\usepackage{amsfonts}
\usepackage{amssymb}
\usepackage{enumitem}
\usepackage{eucal}
\usepackage[all]{xy}
\usepackage{amsxtra}
\usepackage{appendix} 
\usepackage{tensor} 
\usepackage{url}
\usepackage{upgreek}
\usepackage{etoolbox}
\apptocmd{\sloppy}{\hbadness 10000\relax}{}{}
\usepackage{comment}
\usepackage{mathtools}
\usepackage{hyperref}
\allowdisplaybreaks

\RequirePackage{color}
\definecolor{bwgreen}{rgb}{0.183,1,0.5}
\definecolor{bwmagenta}{rgb}{0.7,0.0,0.1}
\definecolor{bwblue}{rgb}{0.317,0.161,1}

\hypersetup{
pdfauthor = {\textcopyright\ Christopher Davis},
pdfcreator = {\LaTeX\ with package \flqq hyperref\frqq},
colorlinks = {true},
linkcolor={bwmagenta},	 
anchorcolor={},	 
citecolor={bwblue},	
filecolor={},	
menucolor={},	
runcolor={},	
urlcolor={bwgreen}, 
}

\DeclareFontFamily{OT1}{rsfs}{}
\DeclareFontShape{OT1}{rsfs}{n}{it}{<-> rsfs10}{}
\DeclareMathAlphabet{\mathscr}{OT1}{rsfs}{n}{it}

\DeclareFontFamily{OT1}{pzc}{}
\DeclareFontShape{OT1}{pzc}{n}{it}{<->s*[2.2]pzc}{}
\DeclareMathAlphabet{\mathpzc}{OT1}{pzc}{b}{sl}

\makeatletter

\newcommand{\Rmnum}[1]{\expandafter\@slowromancap\romannumeral #1@}
\makeatother

\DeclareMathOperator{\Ann}{Ann}

\newcommand*{\FF}{\ensuremath{\mathbf{F}}}   
   
\newcommand*{\CC}{\ensuremath{\mathbf{C}}}              
\newcommand*{\ZZ}{\ensuremath{\mathbf{Z}}}               
\newcommand*{\QQ}{\ensuremath{\mathbf{Q}}}

\newcommand{\OCp}{\mathcal{O}_{\CC_p}} 
         
\numberwithin{equation}{section}
\theoremstyle{plain}
  \newtheorem{theorem}[equation]{Theorem}
    
  \newtheorem{proposition}[equation]{Proposition}
  \newtheorem{lemma}[equation]{Lemma}
  \newtheorem{corollary}[equation]{Corollary}
  \newtheorem{innercustomtheorem}{Theorem}
\newenvironment{customtheorem}[1]
  {\renewcommand\theinnercustomtheorem{#1}\innercustomtheorem}
  {\endinnercustomtheorem}

\theoremstyle{definition}
  \newtheorem{definition}[equation]{Definition}
  \newtheorem{notation}[equation]{Notation}
  \newtheorem{assumption}[equation]{Assumption}

\theoremstyle{remark}
  \newtheorem{example}[equation]{Example}
  \newtheorem{remark}[equation]{Remark}

\begin{document}
\title{On the $p$-typical de\thinspace Rham-Witt complex over $W(k)$}

\author{Christopher Davis}
\address{University of California, Irvine}
\email{daviscj@uci.edu}


\date{\today}

\begin{abstract}
Hesselholt and Madsen in \cite{HM04} define and study the (absolute, $p$-typical) de\thinspace Rham-Witt complex in mixed characteristic, where $p$ is an odd prime.  They give as an example an elementary algebraic description of the de\thinspace Rham-Witt complex over $\ZZ_{(p)}$, $W_{\cdot} \Omega^{\bullet}_{\ZZ_{(p)}}$.  The main goal of this paper is to construct, for $k$ a perfect ring of characteristic~$p > 2$, a Witt complex over $A = W(k)$ with an algebraic description which is completely analogous to Hesselholt and Madsen's description for $\ZZ_{(p)}$.  Our Witt complex is not isomorphic to the de\thinspace Rham-Witt complex; instead we prove that, in each level, the de\thinspace Rham-Witt complex over $W(k)$ surjects onto our Witt complex, and that the kernel consists of all elements which are divisible by arbitrarily high powers of $p$.  We deduce an explicit description of $W_n \Omega^{\bullet}_A$ for each $n \geq 1$.   We also deduce results concerning the de\thinspace Rham-Witt complex over certain $p$-torsion-free perfectoid rings.  
\end{abstract}

\maketitle

\section*{Introduction}

Fix an odd prime $p$ and a $\ZZ_{(p)}$-algebra $R$.  In \cite{HM04}, Hesselholt and Madsen define the (absolute, $p$-typical) de\thinspace Rham-Witt complex over $R$ to be the initial object in the category of Witt complexes over $R$.  Their definition generalizes the de\thinspace Rham-Witt complex of Bloch-Deligne-Illusie, which was defined for $\FF_p$-algebras.  The goal of this paper is to define a Witt complex $E_{\cdot}^{\bullet}$ over $A = W(k)$, where $k$ is a perfect ring of characteristic~$p$, and to use this Witt complex to describe the de\thinspace Rham-Witt complex over $W(k)$ and also to study the de\thinspace Rham-Witt complex over certain perfectoid rings $B$.

Among many other conditions, the de\thinspace Rham-Witt complex $W_{\cdot} \Omega^{\bullet}_R$ is a pro-system of differential graded rings. There is an isomorphism $W_{n}(R) \rightarrow W_{n} \Omega^0_R$, so the degree~zero piece of the de\thinspace Rham-Witt complex is well-understood.  For each positive integer $n$ and for every degree~$d$, there is a surjective morphism of differential graded rings 
\[
\Omega^{d}_{W_n(R)} \twoheadrightarrow W_n\Omega^{d}_R,
\]
and so it is easy to write down elements of $W_{n}\Omega^{d}_R$.   On the other hand, especially in the degree~one case $d = 1$, it is often difficult to determine which of these elements in $W_n \Omega^1_A$ are non-zero.  The author is not aware of a complete algebraic description of the (absolute, $p$-typical) de\thinspace Rham-Witt complex in mixed characteristic for any examples other than $\ZZ_{(p)}$ and polynomial algebras over this ring.  One of the goals of the current paper is to give a complete algebraic description of the de\thinspace Rham-Witt complex over $A = W(k)$, where $k$ is a perfect ring of odd characteristic~$p$.  For example, we prove that in the de\thinspace Rham-Witt complex over $W(k)$, the element $dV^n(1)$ is a non-trivial $p^n$-torsion element for every integer $n \geq 1$.  It is easy to see, using the relation $pdV = Vd$, that this element is indeed $p^n$-torsion, but showing that this element is non-zero takes much more work.

To better analyze relations within the de\thinspace Rham-Witt complex, we first define in Section~\ref{E section} a Witt complex $E_{\cdot}^{\bullet}$ over $A = W(k)$ which has a simple algebraic description as a $W(k)$-module.   The proof that $E_{\cdot}^{\bullet}$ is indeed a Witt complex over $W(k)$ is one of the major parts of this paper.  It is not isomorphic to the de\thinspace Rham-Witt complex over $W(k)$; see Remark~\ref{not iso}.  Instead, in each level $n$ and in each positive degree $d \geq 1$, our Witt complex $E_{\cdot}^{\bullet}$ is the quotient of the de\thinspace Rham-Witt complex by the $W(k)$-submodule consisting of all elements which are divisible by arbitrarily large powers of $p$.  In the language of \cite[Remark~4.8]{Hes15}, our Witt complex $E_{\cdot}^{\bullet}$ is the $p$-typical de\thinspace Rham-Witt complex over $W(k)$ relative to the $p$-typical $\lambda$-ring $(W(k), s_{\varphi})$, where $s_{\varphi}$ is the ring homomorphism $W(k) \rightarrow W(W(k))$ recalled in Proposition~\ref{sF map} below.

Our description of $E_{\cdot}^{\bullet}$, which we define for each $W(k)$ with $k$ a perfect ring of odd characteristic~$p$, is completely modeled after Hesselholt and Madsen's description of $W_n\Omega^1_{\ZZ_{(p)}}$ from \cite[Example~1.2.4]{HM04}.  They show that for all $n \geq 1$, there is an isomorphism of $\ZZ_{(p)}$-modules
\begin{equation} \label{localized equation}
W_n\Omega^1_{\ZZ_{(p)}} \cong \prod_{i = 0}^{n-1} \ZZ/p^i\ZZ \cdot dV^i(1).  
\end{equation}
This shows that $W_n\Omega^1_{\ZZ_{(p)}}$ is non-zero if $n \geq 2$.   The proof in \cite{HM04} involves the topological Hochschild spectrum $T(\ZZ_{(p)})$.  The results below provide an alternative (and elementary) proof  that $W_n\Omega^1_{\ZZ_{(p)}}$ is non-zero if $n \geq 2$.

Of course an elementary algebraic proof of the isomorphism in Equation~(\ref{localized equation}) could be given by directly verifying that the stated groups satisfy all the necessary relations to form a Witt complex.  It is this approach we follow in the current paper for the case $A = W(k)$, where $k$ is a perfect ring of odd characteristic~$p$.  Moreover, we prove that, for such $A$ and for every $n \geq 1$, there is a surjective map
\begin{equation} \label{surj from dRW}
W_n\Omega^1_{A} \twoheadrightarrow \prod_{i = 0}^{n-1} A/p^i A \cdot dV^i(1) =: E_{n}^1,  
\end{equation}
and we prove that the kernel of this map consists of all elements of $W_n\Omega^1_{A}$ which are divisible by arbitrarily large powers of $p$.

The groups $E_n^{\bullet}$ in a Witt complex over $A$ are in particular $W_n(A)$-modules, and the $W_n(A)$-module structure we define is also analogous to the description for $\ZZ_{(p)}$.  In the de\thinspace Rham-Witt complex over $\ZZ_{(p)}$, and in fact in any Witt complex, for integers $i, j \geq 1$, one has 
\begin{equation} \label{V multiplication}
V^j(1) \, dV^i(1) = p^j \,dV^i(1).
\end{equation}  This alone does not completely determine the $W_n(A)$-module structure, but for our specific case $A = W(k)$, there is a ring homomorphism $s_{\varphi}: A \rightarrow W(A)$, and we require that for all $a \in A$ and $x \in E_n^1$, we have $s_{\varphi}(a) x = a \cdot x$.  Here the product $s_{\varphi}(a) x$ on the left side refers to the $W_n(A)$-module structure we wish to define, and the product $a \cdot x$ on the right side refers to the  $A$-module structure on $E_n^1$ that is apparent from the description in Equation~(\ref{surj from dRW}).  This requirement completely determines our $W_n(A)$-module structure.

With these prerequisites in mind, the verification that our complex is a Witt complex is largely straightforward.  The most difficult step is proving that our complex satisfies 
\[
F d [a] = [a]^{p-1} d[a] \in E_n^1
\] 
for every $a \in A$ and for every integer $n \geq 1$.  The difficulty, which arises repeatedly in what follows, lies in the fact that the multiplicative Teichm\"uller lift $[\cdot]: A \rightarrow W(A)$ is not related in a simple way to our ring homomorphism lift $s_{\varphi}: A \rightarrow W(A)$.

Once we know that our complex $E_{\cdot}^{\bullet}$ is a Witt complex over $A$, we attain relatively easily a complete algebraic description of the de\thinspace Rham-Witt complex $W_{\cdot} \Omega^{\bullet}_A$.  See Section~\ref{app to W(k)} for the proofs of the following results, as well as for a more complete (but longer) description of $W_n\Omega^1_A$ (Corollary~\ref{de Rham-Witt as A-module}).

\begin{customtheorem}{A}
Let $k$ denote a perfect ring of odd characteristic~$p$ and let $A = W(k)$. 
\begin{enumerate}
\item Fix an integer $n \geq 1$.  Let $S_n \subseteq W_n\Omega^1_A$ denote $\cap_{j = 1}^{\infty} p^jW_n\Omega^1_A$,  the $W_n(A)$-submodule of all elements which are infinitely $p$-divisible.  Then we have an isomorphism of abelian groups
\[
W_n\Omega^1_A / S_n \cong \prod_{i = 0}^{n-1} A/p^iA.
\]
\item Fix integers $n \geq 1$ and $d \geq 2$.  Then we have an isomporphism of abelian groups
\[
W_n\Omega^d_A \cong \prod_{i = 0}^{n-1} \Omega^d_{A}.
\]
\end{enumerate}
\end{customtheorem}

In Section~\ref{A/xA section}, we turn to describing the de\thinspace Rham-Witt complex over the quotient ring $A/xA$, for an element $x \in A$; this is done with the purpose of applying it in the case that $A/xA$ is a perfectoid ring $B$, and $A = W(B^{\flat})$ is the ring of Witt vectors of the tilt of $B$.  Our complete algebraic description of $W_n \Omega^1_A$ makes extensive use of the ring homomorphisms $s_{\varphi}: A \rightarrow W_n(A)$, and in general we have no such ring homomorphisms $B \rightarrow W_n(B)$, so our algebraic description of $W_n\Omega^1_B$ is less complete.  However, for a certain class of perfectoid rings, we are able to completely describe the kernel of the restriction map $W_{n+1}\Omega^1_B \rightarrow W_n \Omega^1_B$.  We phrase the following theorem in slightly more generality, to include also the case $W(k)$ which is proved earlier.

\begin{customtheorem}{B}
Let $p$ denote an odd prime.  Let $S$ denote either $W(k)$ for $k$ a perfect ring of characteristic~$p$, or else let $S$ denote a $p$-torsion-free perfectoid ring for which there exists some non-zero $p$-power torsion element $\omega \in \Omega^1_S$.  In either of these cases, the following is a short exact sequence of $W_{n+1}(S)$-modules:
\[
0 \rightarrow S \stackrel{(-d, p^n)}{\longrightarrow} \Omega^1_S \oplus S \stackrel{V^n + dV^n}{\longrightarrow} W_{n+1}\Omega^1_S \stackrel{R}{\rightarrow} W_n\Omega^1_S \rightarrow 0.
\]
\end{customtheorem}

See Proposition~\ref{hWn} and Proposition~\ref{B sequence proposition} for the proofs, and also for a description of the module structures.  The existence of an element $\omega$ as described in the statement is guaranteed, for example, whenever $\zeta_p  \in S$ and $d\zeta_p \neq 0$.

One motivation for studying the de\thinspace Rham-Witt complexes we consider in this paper is our hope to adapt results from Hesselholt's paper \cite{Hes06}.  That paper concerns the de\thinspace Rham-Witt complex over the ring of integers in an algebraic closure of a mixed characteristic local field, and we hope to perform a similar analysis in the context of perfectoid rings.  Our proofs for perfectoid rings will be modeled after Hesselholt's proof for $\mathcal{O}_{\overline{\QQ_p}}$, and our proofs will use an induction argument that requires a precise description of the kernel of restriction $W_{n+1}\Omega^1_B \rightarrow W_n \Omega^1_B$.   We will pursue this direction in joint work with  Irakli Patchkoria.

A second, but indirect, motivation for the current paper is the recent remarkable work of Bhatt-Morrow-Scholze in \cite{BMS16}, which makes use of the de\thinspace Rham-Witt complex in mixed characteristic.  Currently this is only a philosophical motivation, however, because they study the \emph{relative} de\thinspace Rham-Witt complex of Langer-Zink \cite{LZ04}, whereas we study the absolute de\thinspace Rham-Witt complex of Hesselholt-Madsen \cite{HM03, HM04, Hes05}.  Our work is not directly relevant to the work of Bhatt-Morrow-Scholze, but it could potentially be relevant to generalizations of their work which involved the \emph{absolute} de\thinspace Rham-Witt complex.  

\subsection{Notation} Throughout this paper, $p > 2$ denotes an odd prime, $k$ is a perfect ring of characteristic~$p$, $W$ denotes $p$-typical Witt vectors, and $A = W(k)$.  To distinguish between the Witt vector Frobenius on $A = W(k)$ and on $W(A)$, we write $\varphi$ for the Witt vector Frobenius on $A$ and we write $F$ for the Witt vector Frobenius on $W(A)$ and on $W_{\cdot} \Omega^{\bullet}_A$.  Rings in this paper are assumed to be commutative and to have unity, and ring homomorphisms are assumed to map unity to unity.  We write $\Omega^1_{R}$ for the $R$-module of absolute K\"ahler differentials, i.e., $\Omega^1_R = \Omega_{R/\ZZ}$ in the notation of \cite[Section~25]{Mat89}.  The de\thinspace Rham-Witt complex we consider is the absolute, $p$-typical de\thinspace Rham-Witt complex defined in \cite[Introduction]{HM04}.  

\subsection{Acknowledgments} This paper would not exist without the regular meetings I had with Lars Hesselholt during my two years in Copenhagen as his postdoc.  Also extremely helpful were visits to Copenhagen by Bhargav Bhatt and Matthew Morrow; in particular, the arguments in Section~\ref{results A} involving the de\thinspace Rham complex were shown to me by Bhargav Bhatt.  Thanks also to Bryden Cais, Kiran Kedlaya, Irakli Patchkoria, and David Zureick-Brown for help throughout the project.  I am also grateful to Bhargav Bhatt, Lars Hesselholt, and Kiran Kedlaya for feedback on earlier versions of this paper.  Finally, I thank the anonymous referee for a careful reading of an earlier version of this paper.

\section{Background on Witt complexes and the de\thinspace Rham-Witt complex}

Fix $k$, a perfect ring of odd characteristic~$p$ and let $A = W(k)$.  The main goal of this paper is to construct a certain Witt complex over $A$, and to use this Witt complex to deduce properties of the de\thinspace Rham-Witt complex over $A$.  Similar properties are proven in the work of Hesselholt \cite{Hes05, Hes06} and Hesselholt-Madsen \cite{HM03, HM04}; the main difference between our results and these earlier results is that our proofs use only algebra.  The only aspect of the current paper which is not elementary is our proof that $\Omega^1_{W(k)}$ has no non-trivial $p$-torsion (Proposition~\ref{Omega bijection}), which uses the cotangent complex.  The current paper does not use any notions from algebraic topology, such as the spectrum $TR_{\bullet}^{\cdot}$.  

The current paper does, however, use many standard facts about ($p$-typical) Witt vectors $W(R)$ and the ($p$-typical, absolute) de\thinspace Rham-Witt complex $W_{\cdot} \Omega^{\bullet}_R$, and it is written with the assumption that the reader is familiar with their basic properties, including the case $R$ is not characteristic~$p$.  For background on Witt vectors, we refer to \cite{Ill79} or to the brief introduction given in Section~1 of \cite{HM04}.  A thorough treatment of Witt vectors is given in Section~1 of \cite{Hes15}, but those results are framed in the context of big Witt vectors instead of $p$-typical Witt vectors.  

We work in this section over an arbitrary $\ZZ_{(p)}$-algebra $R$, where $p$ is an odd prime.  We now recall the basic properties of Witt complexes and the de\thinspace Rham-Witt complex which we will use.  Our reference is \cite{HM04}.  

The de\thinspace Rham-Witt complex over $R$ (or, more generally, any Witt complex over $R$) is a pro-system of differential graded rings.  The index indicating the position in the pro-system is a positive integer $n = 1, 2, \ldots$ which we refer to as the \emph{level}.  The index indicating the degree in the differential graded ring is a non-negative integer $d = 0, 1, \ldots$ which we refer to as the \emph{degree}.  We write $E_{n}^{d}$ for the level~$n$, degree~$d$ component of a Witt complex $E^{\bullet}_{\cdot}$.  

\begin{definition}[{\cite[Introduction]{HM04}}] Fix an odd prime $p$ and a $\ZZ_{(p)}$-algebra $R$.  A \emph{Witt complex} over $R$ is the following.
\begin{enumerate}
\item A pro-differential graded ring $E^{\bullet}_{\cdot}$ and a strict map of pro-rings 
\[
\lambda: W_{\cdot}(R) \rightarrow E^{\bullet}_{\cdot}.
\]
\item A strict map of pro-graded rings 
\[
F: E_{\cdot}^{\bullet} \rightarrow E_{\cdot - 1}^{\bullet}
\]
such that $F \lambda = \lambda F$ and for all $r \in R$, we have
\[
F d \lambda\left([r]\right) = \lambda\left([r]^{p-1}\right) d\lambda\left([r]\right).
\]
\item A strict map of graded $E_{\cdot}^{\bullet}$-modules
\[
V: F_*E_{\cdot-1}^{\bullet} \rightarrow E_{\cdot}^{\bullet}.
\]
\big(In other words,
\[
V( F(\omega) \eta) = \omega V(\eta) \text{ for all } \omega \in E_{\cdot}^{\bullet}, \eta \in E_{\cdot-1}^{\bullet},
\]
and similarly for multiplication on the right.\big) The map $V$ must further satisfy $V \lambda = \lambda V$ and 
\[
FdV = d, \qquad FV = p.
\]
\end{enumerate}
\end{definition}

\begin{remark}
In this paper we never consider the prime $p = 2$.  See \cite[Definition~4.1]{Hes15} for a definition of Witt complex which can be used for all primes, or \cite{Cos08} for a careful treatment of the $2$-typical de\thinspace Rham-Witt complex.  One subtlety is that for $p = 2$, the differential does not necessarily satisfy $d \circ d = 0$.
\end{remark}

The following theorem defines the de\thinspace Rham-Witt complex over $R$ as the initial object in the category of Witt complexes over $R$.  Its existence is proved in \cite{HM04}.

\begin{theorem}[{\cite[Theorem~A]{HM04}}]
Let $R$ denote a $\ZZ_{(p)}$-algebra, where $p$ is an odd prime.  There is an initial object $W_{\cdot} \Omega^{\bullet}_R$ in the category of Witt complexes over $R$.  We call this complex the de\thinspace Rham-Witt complex over $R$.  Moreover, for every $d \geq 0$ and $n \geq 1$, the canonical map
\[
\Omega^{d}_{W_n(R)} \rightarrow W_n\Omega^d_R
\]
is surjective.
\end{theorem}

The following result, like our last result, is proved in \cite{HM04}. It describes the degree~0 piece and the level~1 piece of the de\thinspace Rham-Witt complex, respectively.

\begin{theorem}
Let $R$ denote a $\ZZ_{(p)}$-algebra, where $p$ is an odd prime.
\begin{enumerate}
\item \cite[Remark~1.2.2]{HM04} The canonical map $\lambda: W_n(R) \rightarrow W_n \Omega^0_R$ is an isomorphism for all~$n \geq 1$.
\item \cite[Theorem~D and the first sentence of the proof of Proposition~5.1.1]{HM04} The canonical map $\Omega^{\bullet}_R \rightarrow W_1 \Omega^{\bullet}_R$ is an isomorphism.
\end{enumerate}
\end{theorem}

Two of the main results of this paper are Proposition~\ref{hWn} and Proposition~\ref{B sequence proposition} below.  The main content of these propositions describes, for suitable rings $R$, the intersection 
\[
V^n \left( \Omega^1_R \right) \cap dV^n(R)  \subseteq W_{n+1} \Omega^1_R.
\]  
Our next proposition, which is true for every $\ZZ_{(p)}$-algebra $R$, identifies  
\[
V^n \left( \Omega^1_R \right) + dV^n(R)  \subseteq W_{n+1} \Omega^1_R
\] 
as the kernel of restriction.

\begin{proposition} \label{Fil}
Let $R$ denote a $\ZZ_{(p)}$-algebra, where $p$ is an odd prime.  Fix integers $d \geq 1$ and $n \geq 1$.  Then $\omega$ is in the kernel of restriction
\[
W_{n+1}\Omega^d_R \rightarrow W_{n} \Omega^d_R
\]
if and only if there exist $\alpha \in \Omega^{d}_R$ and $\beta \in \Omega^{d-1}_R$ such that
\[
\omega = V^n(\alpha) + dV^n(\beta).
\]
\end{proposition}

The difficult part is the \emph{only if} direction.  See \cite[Lemma~3.2.4]{HM03} for a proof in terms of the log de\thinspace Rham-Witt complex.  We recall the idea of that proof.  (See also the proof of Proposition~\ref{Witt complex over B} below for similar arguments.)  For every $n, d$, define
\[
^{\prime} W_n \Omega^d_R := W_{n+1}\Omega^d_R / \big( V^n(\Omega^d_R) + dV^n(\Omega^{d-1}_R) \big). 
\]
One then shows that $^{\prime} W_{\cdot} \Omega^{\bullet}_R$ is an initial object in the category of Witt complexes over $R$, and hence in particular that the natural map 
\begin{equation} \label{prime W}
^{\prime} W_n \Omega^d_R \rightarrow W_n \Omega^d_R
\end{equation}
is an isomorphism.  That natural map is induced by restriction $W_{n+1}\Omega^d_R \rightarrow W_n \Omega^d_R$, so our proposition follows from the injectivity of the map in Equation~(\ref{prime W}).

The following results we recall from \cite{HM04} have significantly easier proofs than the previous results we have cited; the proofs of the relations in Proposition~\ref{Witt complex relations} below are just a few lines of computation.

\begin{proposition}[{\cite[Lemma~1.2.1]{HM04}}]  \label{Witt complex relations} Again let $R$ denote a $\ZZ_{(p)}$-algebra, where $p$ is an odd prime.  The following equalities hold in every Witt complex over $R$:
\[
dF = pFd, \qquad Vd = pdV, \qquad V(x_0 dx_1 \cdots dx_m) = V(x_0) dV(x_1) \cdots dV(x_m).
\]

\end{proposition}

\section{Results on $W(A)$ and $\Omega^1_A$ when $A = W(k)$} \label{results A}

Let $k$ denote a perfect ring of odd characteristic~$p$ and let $A = W(k)$.  In this paper, we study the de\thinspace Rham-Witt complex over $A$.  In this section, we prove several preliminary results about the degree~zero case, $W(A)$, and the level~one case, $\Omega^1_A$.   Special thanks are due to Bhargav Bhatt and Lars Hesselholt for their assistance with the $\Omega^1_A$ proofs.  

The following result allows us to view the ring $W(A)$ as an $A$-algebra.  This is a key fact.  This is also a similarity between the case $A = W(k)$ and the case $A = \ZZ_{(p)}$, after which our results are modeled: the ring $W(A)$ is an $A$-algebra and the ring $W(\ZZ_{(p)})$ is a $\ZZ_{(p)}$-algebra.  This is also the main reason our methods don't easily translate to more general rings such as ramified extensions of $\ZZ_p$.  

Recall that, to avoid confusion, we write the Witt vector Frobenius differently on $A = W(k)$ from how we write it on $W(A) = W(W(k))$: we write $\varphi: A \rightarrow A$ and $F: W(A) \rightarrow W(A)$ for these Witt vector Frobenius maps.  The map $\varphi$ is a ring isomorphism, but the map $F$ is not an isomorphism.

\begin{proposition}[{\cite[(0.1.3.16)]{Ill79}}] \label{sF map}
Let $k$ denote a perfect ring of characteristic~$p$, let $A = W(k)$, and let $\varphi: A \rightarrow A$ denote the Witt vector Frobenius.  Then there is a unique ring homomorphism 
\[
s_{\varphi}: A \rightarrow W(A)
\]
satisfying $F \circ s_{\varphi} = s_{\varphi} \circ \varphi$ and such that for all $a \in A$, the ghost components of $s_{\varphi}(a)$ are $(a, \varphi(a), \varphi^2(a), \ldots)$.  
\end{proposition} 

\begin{proof}
The ring $A$ is $p$-torsion free, so this result follows from \cite[(0.1.3.16)]{Ill79}, provided we know that the ring homomorphism $\varphi: A \rightarrow A$ satisfies $\varphi(a) \equiv a^p \bmod pA$ for all $a \in A$.  This last congruence is in fact true more generally for any ring $W(R)$ of $p$-typical Witt vectors.  We recall the short proof from \cite[Section 0.1.4]{Ill79}. For arbitrary $a \in W(R)$, write $a = [r_0] + V(a_+)$, where $r_0 \in R$ and $a_+ \in W(R)$.  We then have 
\begin{align*}
\varphi(a) &= [r_0]^p + pa_{+} \\
& \equiv [r_0]^p \bmod pW(R)\\
& \equiv \left( [r_0] + V(a_+) \right)^p \bmod pW(R),
\end{align*}
where the last congruence uses that $V(x) V(y) = pV(xy) \in pW(R)$ for Witt vectors $x,y \in W(R)$.
\end{proof}

\begin{lemma} \label{V-combination lemma}
For every $x \in W(A)$, there exist unique elements $a_0, a_1, \ldots \in A$ for which 
\begin{equation} \label{V formula}
x = \sum_{i = 0}^{\infty} s_{\varphi}(a_i) V^i(1) \in W(A).
\end{equation} 
\end{lemma}

\begin{proof}
We have 
\[
\sum_{i = 0}^{\infty}  s_\varphi(a_i) V^i(1) = \sum_{i = 0}^{\infty}  V^i (F^i (s_\varphi(a_i) )) = \sum_{i = 0}^{\infty}  V^i (s_\varphi(\varphi^i(a_i)) ),
\]
so the result now follows from the fact that $\varphi: A \rightarrow A$ is an isomorphism and that the first component of $s_{\varphi}(a) \in W(A)$ is $a$.  
\end{proof} 

\begin{lemma}
If $x \in W(A)$ is given as in Equation~(\ref{V formula}), then 
\[
V(x) = \sum_{i = 1}^{\infty} s_{\varphi}(\varphi^{-1}(a_{i-1})) V^i(1) \in W(A).
\]
\end{lemma}

\begin{proof}
This follows from the formula $V(F(x)y) = x V(y)$, for $x,y \in W(A)$ and from the fact that $F(s_{\varphi}(a_i)) = s_{\varphi}(\varphi(a_i))$.  
\end{proof}

The following result gives explicit formulas for the elements $a_i \in A$ appearing in Equation~(\ref{V formula}) in the specific case that $x$ is a Teichm\"uller lift of some element $a \in A$.  The main technical difficulty of this paper involves studying congruences involving these coefficients.  

\begin{lemma} \label{Teich lemma}
In the specific case $x = [a] \in W(A)$ is the Teichm\"uller lift of an element $a \in A$, then the terms $a_i$ from Equation~(\ref{V formula}) are given by the formulas
 $a_0 = a$ and $a_i = \varphi^{-i}\left( \frac{a^{p^i} - \left(\varphi(a)\right)^{p^{i-1}}}{p^i}\right) $ for $i \geq 1$.  
\end{lemma} 

\begin{proof}
This follows using induction on $i$, by comparing the ghost components of the two sides of Equation~(\ref{V formula}).  (Notice that the ghost map is injective because $A$ is $p$-torsion free.)  To simplify the proof, notice that a finite sum
\[
s_{\varphi}(a_0) + s_{\varphi}(a_1) V(1) + \cdots + s_{\varphi}(a_n) V^n(1),
\]
has ghost components which stabilize in the following pattern 
\[
(w_0, \ldots, w_{n-1}, w_n, \varphi(w_n), \varphi^2(w_n), \ldots ).  
\]
\end{proof}

When we define our Witt complex $E_{\cdot}^{\bullet}$ in Section~\ref{E section}, we will express $E_n^1$ in terms of quotients $A/p^i A$.  The groups $E_n^1$ in a Witt complex over $A$ always possess a $W_n(A)$-module structure, and the following lemma describes the $W_n(A)$-module structure we put on $A/p^i A$; notice that this module structure is not the one induced by the obvious projection map $W_n(A) \rightarrow A$.  

\begin{lemma} \label{module structure En}
Let $n, i \geq 1$ be integers and consider the map $W_n(A) \rightarrow A/p^i A$ given by 
\[
\sum_{j = 0}^{n-1} s_{\varphi}(a_j) V^j(1) \mapsto  \sum_{j = 0}^{n-1} a_j p^j.
\]
This is a surjective ring homomorphism with kernel the ideal in $W_n(A)$ generated by 
\[
\{ p^i, V^j(1) - p^j \mid 0 \leq j \leq n-1\}.
\] 
\end{lemma}

\begin{proof}
If we view $W_n(A)$ as an $A$-module via $s_{\varphi}$, then it's clear that the map is a surjective $A$-module homomorphism.  To prove it's a ring homomorphism, we use the formula $V^j(1) V^i(1) = p^j V^i(1)$ for $j \leq i$.

We now prove the statement about the kernel.  Clearly the proposed elements are in the kernel; we now show an arbitrary element in the kernel is generated by the proposed elements.  Assume $ \sum_{j = 0}^{n-1} s_{\varphi}(a_j) V^j(1) $ is in the kernel.  This means that there exists $a \in A$ such that
\begin{align*}
p^ia &= \sum_{j = 0}^{n-1} a_j p^j \in A. \\
\intertext{Applying $s_{\varphi}$ to both sides, we find }
p^is_{\varphi}(a) &= \sum_{j = 0}^{n-1} s_{\varphi}(a_j) p^j \in W_n(A), \\
\intertext{and thus }
p^is_{\varphi}(a) + \sum_{j = 0}^{n-1} s_{\varphi}(a_j) (V^j(1) - p^j) &= \sum_{j = 0}^{n-1} s_{\varphi}(a_j) V^j(1),
\end{align*}
which completes the proof.
\end{proof}

This concludes our collection of preliminary results on Witt vectors over $A = W(k)$.  We now turn our attention to $\Omega^1_{W(k)}$.  We thank Bhargav Bhatt and Lars Hesselholt for their help with the remainder of this section.  Our first result, Proposition~\ref{Omega bijection}, is the most important.  It says that multiplication by $p$ is bijective on $\Omega^1_{W(k)}$; we will use this result repeatedly.  By contrast, the results from Proposition~\ref{de Rham field} to the end of this section are closer to ``reality-checks".  For example, Corollary~\ref{de Rham non-zero} below shows that $\Omega^1_{W(k)}$ is not the zero-module.  

\begin{proposition} \label{Omega bijection}
Let $k$ denote a perfect ring of characteristic~$p$.  Then multiplication by $p$: $\Omega^1_{W(k)} \rightarrow \Omega^1_{W(k)}$ is a bijection.  
\end{proposition} 

\begin{remark}
The proof below is due to Bhargav Bhatt.  The tools used in the proof (the cotangent complex and, more generally, the language of derived categories) do not appear elsewhere in this paper, so the reader (or author) who is not comfortable with them is advised to treat the proof of Proposition~\ref{Omega bijection} as a black box.  See also the proof of \cite[Lemma~2.2.4]{HM03} for a proof of a related result.

Before giving Bhatt's proof, we point out an elementary argument for surjectivity. The Witt vector Frobenius $\varphi: W(k) \rightarrow W(k)$ is surjective on one hand, and on the other hand, $\varphi(a) \equiv a^p \bmod pW(k)$ for every $a \in W(k)$.  So for every $a \in W(k)$, we can find $a_0, a_1 \in W(k)$ such that $a = a_0^p + pa_1$.  Thus every $da \in \Omega^1_{W(k)}$ is divisible by $p$, and hence multiplication by $p$ on $\Omega^1_{W(k)}$ is surjective.  We are not aware of a similarly elementary proof of injectivity.
\end{remark}

\begin{proof}
Let $L_{W(k)/\ZZ}$ denote the cotangent complex.  Because $\ZZ \rightarrow W(k)$ is flat, we have
\[
L_{W(k)/\ZZ} \otimes_{\ZZ}^{L} \FF_p \cong L_{k/\FF_p}  
\]
by \cite[\href{http://stacks.math.columbia.edu/tag/08QQ}{Tag 08QQ}]{stacks-project}.
The right-hand side is zero, because the Frobenius automorphism on $k$ induces a map on $L_{k/\FF_p}$ which is simultaneously zero and an isomorphism.  Thus the left-hand side is also 0.  This implies that multiplication by $p$ on $L_{W(k)/\ZZ}$ is a quasi-isomorphism.  In particular, multiplication by $p$ is an isomorphism on $H^0(L_{W(k)/\ZZ}) \cong \Omega^1_{W(k)}$, which completes the proof. 
\end{proof}

Throughout this paper, $k$ denotes a perfect \emph{ring} of characteristic~$p$.  We prove Corollary~\ref{de Rham non-zero} below for $W(k)$ by deducing it from Proposition~\ref{de Rham field}, which concerns the case of $W(k')$, where $k'$ is a perfect \emph{field} of characteristic~$p$.

\begin{proposition} \label{de Rham field}
Let $k'$ denote a perfect field of characteristic~$p$.  Let $\{x_{\alpha}\}_{\alpha \in A} \subseteq W(k')$ denote elements such that $\{ x_{\alpha} \}_{\alpha \in A}$ is a transcendence basis for $W(k')[1/p]$ over $\QQ$.  Then $\{ dx_{\alpha} \}_{\alpha \in A}$ is a basis for $\Omega^1_{W(k')}$ as a $W(k')[1/p]$-vector space.
\end{proposition} 

\begin{proof}
By Proposition~\ref{Omega bijection}, we have
\[
\Omega^1_{W(k')} \cong \Omega^1_{W(k')} \otimes_{W(k')} W(k')[1/p] \cong \Omega^1_{W(k')[1/p]} \cong \Omega^1_{W(k')[1/p]/\QQ}.
\]
Thus it suffices to prove that if $\{ x_{\alpha} \}_{\alpha \in A}$ is a transcendence basis for a field $K/\QQ$, then $\{ dx_{\alpha} \}_{\alpha \in A}$ is a $K$-basis for $\Omega^1_{K/\QQ}$.  The result now follows by \cite[Theorem~26.5]{Mat89}. 
\end{proof} 

\begin{corollary} \label{de Rham non-zero}
Let $k$ denote a perfect ring of characteristic~$p$.  Then the $W(k)$-module $\Omega^1_{W(k)}$ is non-zero.
\end{corollary} 

\begin{proof}
Let $\mathfrak{m} \subseteq k$ denote a maximal ideal.  Then $k \rightarrow k/\mathfrak{m}$ is a surjection from $k$ onto a perfect field of characteristic~$p$; write $k' = k/\mathfrak{m}$.  The induced map $W(k) \rightarrow W(k')$ is a surjective ring homomorphism, so $\Omega^1_{W(k)} \rightarrow \Omega^1_{W(k')}$ is a surjective $W(k)$-module homomorphism.  Because $W(k')$ is uncountable, the field $W(k')[1/p]$ is transcendental over $\QQ$, so our result follows from Proposition~\ref{de Rham field}.
\end{proof}

\begin{corollary}
For every integer $n \geq 1$, the $W_n(W(k))$-module $W_n\Omega^1_{W(k)}$ is non-zero.
\end{corollary}

\begin{proof}
Begin with any non-zero element $\alpha \in \Omega^1_{W(k)}$.  We then have $p^{n-1}\alpha \neq 0$ by Proposition~\ref{Omega bijection}, but on the other hand, $p^{n-1}\alpha = F^{n-1} V^{n-1}(\alpha)$, and so $V^{n-1}(\alpha) \in W_n\Omega^1_{W(k)}$ is non-zero.
\end{proof}

\section{A $p$-adically separated Witt complex over $W(k)$} \label{E section}

Let $k$ denote a perfect ring of odd characteristic~$p$ and let $A = W(k)$.  
We are going to define a Witt complex over $A$.  Our definition is modeled after \cite[Example~1.2.4]{HM04}, which gives a completely analogous description of the de\thinspace Rham-Witt complex over $\ZZ_{(p)}$.  

As an abelian group, we define 
\begin{align*}
E_n^0 &:= W_n(A) \text{ for all } n \geq 1\\
E_n^1 &:= \prod_{i = 0}^{n-1} A/p^iA \cdot dV^i(1) \text{ for all } n \geq 1\\
E_n^d &:= 0 \text{ for all } n \geq 1,\, d \geq 2;
\end{align*}
here $dV^i(1)$ should be viewed as a formal basis symbol.
The ring structure on $E_n^{\bullet}$ is obvious with the exception of the multiplication $E_n^0 \times E_n^1 \rightarrow E_n^1$, and for this we use the ring homomorphisms from Lemma~\ref{module structure En} to give $A/p^iA$ the structure of a $W_n(A)$-module. (We note again that the module structure does not arise from the restriction map $W_n(A) \rightarrow W_1(A) = A$.)   Define $\lambda: W_{n}(A) \rightarrow E_n^0$ to be the identity map and equip $E_n^0$ with the usual ring structure and with the usual maps $R, F, V$.  

Recalling Lemma~\ref{V-combination lemma}, which guarantees that each element in $W_n(A)$ has a unique expression $\sum_{i = 0}^{n-1} s_{\varphi}(a_i) \cdot V^i(1)$, we define $d: E_n^0 \rightarrow E_n^1$ by the formula 
\begin{align*}
d\left( \sum_{i = 0}^{n-1} s_{\varphi}(a_i) V^i(1) \right) &= \sum_{i = 1}^{n-1} a_i \cdot dV^i(1).  \\
\intertext{Define $R: E_{n+1}^1 \rightarrow E_{n}^1$ by the formula}
R\left( \sum_{i = 0}^{n} a_i \cdot dV^i(1) \right) &=\sum_{i = 0}^{n-1} a_i  \cdot dV^i(1).\\
\intertext{Define $F: E_{n+1}^1 \rightarrow E_n^1$ by the formula}
F\left( \sum_{i = 1}^{n} a_i  \cdot dV^i(1) \right) &= \sum_{i = 0}^{n-1} \varphi(a_{i+1})  \cdot dV^i(1). \\
\intertext{Define $V: E_{n}^1 \rightarrow E_{n+1}^1$ by the formula}
V\left( \sum_{i = 1}^{n-1} a_i \cdot dV^i(1) \right) &= \sum_{i = 1}^{n-1}  p\varphi^{-1} (a_{i})  \cdot dV^{i+1}(1).
\end{align*}
We emphasize that this last definition means in particular that $V\left( dV^i(1) \right) = p \cdot dV^{i+1}(1)$.  

\begin{remark}
We use the dot $\cdot$ in the notation $A/p^iA \cdot dV^i(1)$ to help distinguish between this $A/p^iA$-module structure and the $W_n(A)$-module structure, which we write without the dot.  For example, if we let $\pi_{n,i}: W_n(A) \rightarrow A/p^iA$ denote the ring homomorphism from Lemma~\ref{module structure En}, then we would write $x dV^i(1) = \pi_{n,i}(x) \cdot dV^i(1)$.  This distinction isn't mathematically important, but we find it helps to reinforce whether we are multiplying by elements in $A/p^iA$ or by elements in $W_n(A)$ or $W(A)$.   
\end{remark}

Before proving that $E_{\cdot}^{\bullet}$ is a Witt complex, we make a preliminary calculation that does not involve Witt vectors.
This calculation will be used to verify that 
\begin{equation} \label{Teich notation}
Fd([a]) = [a]^{p-1} d([a]) \in E_n^1  
\end{equation}
holds for all $n \geq 1$, which is the most difficult step in our verification that $E_{\cdot}^{\bullet}$ is a Witt complex.

\begin{remark}
In Equation~(\ref{Teich notation}), we are being less careful with notation than Hesselholt and Madsen are in \cite{HM04}.  In their notation, this equation would be written 
\[
Fd([a]_{n+1}) = ([a]_n)^{p-1} d([a]_n) \in E_n^1,
\]
where the subscripts are indicating $[a]_n \in W_n(A)$ and $[a]_{n+1} \in W_{n+1}(A)$.
\end{remark}

\begin{lemma} \label{congruence F lemma}
Continue to let $A = W(k)$, where $k$ is a perfect ring of odd characteristic~$p$, and let $\varphi: A \rightarrow A$ denote the Witt vector Frobenius.  Fix $a \in A$.  Then for every $i \geq 1$, we have 
\begin{equation} \label{congruence F equation}
\frac{1}{p^{i+1}} \left( a^{p^{i+1}} - \varphi(a)^{p^i} \right) \equiv \frac{1}{p^{i}} \left( a^{p^{i}} - \varphi(a)^{p^{i-1}} \right) a^{p^i (p-1)} \bmod p^i A.
\end{equation}
\end{lemma}  

\begin{proof}
The only fact we will use about $\varphi: A \rightarrow A$ is that for every $a \in A$, there exists $x \in A$ such that $\varphi(a) = a^p + px$.  Multiplying both sides of (\ref{congruence F equation}) by $p^{i+1}$ and applying the binomial theorem to the powers of $\varphi(a) = a^p + px$, we reduce immediately to proving that 
\begin{align*}
\sum_{j = 1}^{p^{i}} \binom{p^i}{j} (a^p)^{p^{i}-j} (px)^j &\equiv pa^{p^i(p-1)} \sum_{j = 1}^{p^{i-1}} \binom{p^{i-1}}{j} (a^p)^{p^{i-1}-j} (px)^j \bmod p^{2i+1} A. \\
\intertext{By distributing the $a^{p^i(p-1)}$ term on the right side, this simplifies to proving that }
\sum_{j = 1}^{p^{i}} \binom{p^i}{j} a^{p^{i+1}-pj} (px)^j &\equiv p \sum_{j = 1}^{p^{i-1}} \binom{p^{i-1}}{j} a^{p^{i+1}-pj} (px)^j \bmod p^{2i+1} A. 
\end{align*}
By comparing the coefficients of the $a^m x^n$ monomials, it suffices then to prove the following two claims:
\begin{itemize}
\item For every $j$ in the range $1 \leq j \leq p^{i-1}$, we have 
\[
p^j \binom{p^i}{j} \equiv p^{j+1} \binom{p^{i-1}}{j} \bmod p^{2i + 1}.
\]
\item For every $j$ in the range $p^{i-1} + 1 \leq j \leq p^i$, we have 
\[
p^j\binom{p^i}{j} \equiv 0 \bmod p^{2i + 1}.
\]
\end{itemize}

To prove the first claim, we rewrite it as
\begin{align*}
p^j \left( \binom{p^i}{j} -  p\binom{p^{i-1}}{j}\right) &\equiv 0 \bmod p^{2i + 1}. \\
\intertext{The left side equals 0 if $j = 1$, so we may assume $j \geq 2$ and simplify the expression as }
p^j \frac{p^i}{j!} \bigg( (p^i - 1) \cdots (p^i - j + 1) - (p^{i-1}-1)\cdots (p^{i-1} -j+1)\bigg) &\equiv 0 \bmod p^{2i+1}.
\end{align*}
The term inside the parentheses is the difference of two terms which are congruent modulo~$p^{i-1}$, hence the term inside the parentheses is divisible by $p^{i-1}$.  Thus it suffices to show that for every $j \geq 2$ we have 
\[
p^j \frac{p^{2i - 1}}{j!} \equiv 0 \bmod p^{2i + 1}.
\] 
Thus it suffices to show that for every $j \geq 2$, we have $j - v_p(j!) \geq 2$, where $v_p$ denotes the $p$-adic valuation.  Because $p \geq 3$, the inequality is true if $j = 2$.  For the case $j \geq 3$, again using $p \geq 3$, we compute
\[
j - v_p(j!) \geq j - \left( \frac{j}{p} + \frac{j}{p^2} + \cdots \right) = j - j \frac{1}{p(p-1)} \geq j - \frac{j}{6} = \frac{5j}{6} \geq \frac{15}{6} \geq 2,
\]
which completes the proof of the first claim. 

To prove the second claim, we first treat the case $j = p^i$.  Then we need to show that $p^i \geq 2i + 1$, which is true because $p \geq 3$ and $i \geq 1$.  For the case $p^{i-1} + 1 \leq j < p^i$, we know the binomial coefficient in the expression has $p$-adic valuation at least one, so it suffices to prove that $j+1 \geq 2i + 1$.  Thus it suffices to prove that $p^{i-1} + 2 \geq 2i + 1$.  Again this holds because $p \geq 3$ and $i \geq 1$.  
\end{proof}

\begin{remark}
Lemma~\ref{congruence F lemma} is false in general if $p = 2$.  For example, it is already false in the case $A = \ZZ_2$, $\varphi = \text{id}$, $a = 2$, and $i = 1$.
\end{remark}

We can now state our main theorem of this section; all the main results of this paper are dependent on the following result.  

\begin{theorem}
Let $k$ be a perfect ring of characteristic $p > 2$, and let $A = W(k)$.  The complex $E_{\cdot}^{\bullet}$ defined above is a Witt complex over $A$.  
\end{theorem}

\begin{proof}
Many of the required properties are obvious; the main difficulty is proving that for all $a \in A$ and all $n \geq 2$, we have 
\begin{equation} \label{Teich relation}
Fd([a]) = [a]^{p-1} d([a]) \in E_{n-1}^1.  
\end{equation}
We postpone this verification to the end of the proof.  

The following properties are clear: 
\begin{itemize}
\item For each $n$, $E_n^{\bullet}$ is a ring.
\item The maps $R$ are ring homomorphisms.
\item The map $\lambda$ is a ring homomorphism that commutes with $R$.
\item The maps $F, V$ commute with $\lambda$.  
\item The maps $d, F, V$ are additive.
\item The maps $d, F, V$ commute with $R$.
\item The composition $FV$ is equal to multiplication by $p$.
\end{itemize}

Next we check that $d$ verifies the Leibniz rule.  Because $d$ is additive and because $ds_{\varphi}(a) = 0$ for all $a$, it suffices to prove that for all $1 \leq j \leq i$, we have
\[
d\left( V^i(1) V^j(1)\right) = V^i(1)  dV^j(1) + V^j(1) dV^i(1).
\]
Using the definition of our multiplication $E_n^0 \times E_n^1 \rightarrow E_n^1$ and using $V(x)V(y) = pV(xy)$, we see that both sides are equal to $(p^j + p^iA) \cdot dV^i(1)$.  

Next we check that $F$ is multiplicative.  The only part which isn't obvious is to show that if $x \in E_n^0$ and $y \in E_n^1$, then we have 
\[
F(xy) = F(x) F(y).
\]
Because we already know $F$ is additive, it suffices to check this in the special cases $x = x_1 := s_{\varphi}(a)$, $x = x_2 := V^i(1)$ with $i \geq 1$, and $y = (b + p^j A) \cdot dV^j(1),$ where $j \geq 1$.  We have $F(x_1) = s_{\varphi}(\varphi(a))$, $F(x_2) = pV^{i-1}(1)$, and $F(y) = (\varphi(b) + p^{j-1}A) \cdot dV^{j-1}(1)$.  On the other hand, $x_1 y = ( ab + p^jA) \cdot dV^j(1)$ and $F(x_1y) = (\varphi(ab) + p^{j-1}A) \cdot dV^{j-1}(1) = F(x_1) F(y)$.  We also have $x_2 y = (p^i b+ p^jA) \cdot dV^j(1)$ and $F(x_2 y) = (p^i \varphi(b) + p^{j-1}A) \cdot dV^{j-1}(1) = F(x_2) F(y)$. 

We next check that for all $x \in E_{n+1}^{\bullet}$ and $y \in E_{n}^{\bullet}$, we have
\begin{equation} \label{V module}
V(F(x) y) = x V(y).
\end{equation}
This is obvious if $x,y$ are both in degree~zero or both in degree~one, thus we only need to consider the case that one of them is degree~zero and the other is degree~one.  It suffices to consider the case that the degree~one term has the form $dV^j(1)$ and the degree zero term has the form $s_{\varphi}(a) V^i(1)$.  If $x = dV(1)$ and $y = s_{\varphi}(a) V^i(1)$, then both sides of Equation~(\ref{V module}) are zero.  If $x = dV^j(1)$ with $j \geq 2$ and  $y = s_{\varphi}(a) V^i(1)$, we compute 
\begin{align*}
V\bigg(F(x)y\bigg) &= V\bigg( p^i a  \cdot dV^{j-1}(1) \bigg) \\
&= p^{i+1} \varphi^{-1}(a)   \cdot dV^{j}(1) \\
&= \bigg( s_{\varphi}\left(\varphi^{-1}(a)\right) V^{i+1}(1)\bigg)  dV^j(1) \\
&= xV(y).  
\end{align*}
If $x = s_{\varphi}(a)$ and $y = dV^j(1)$, then we compute 
\[
V(F(x)y) = V\bigg(\varphi(a) \cdot dV^j(1) \bigg) = ps_{\varphi}(a) dV^{j+1}(1) = x V(y).
\]
If $x = s_{\varphi}(a) V^i(1)$ with $i \geq 1$ and $y = dV^j(1)$, then we compute
\begin{align*}
V(F(x)y) &= V\bigg(s_{\varphi}(\varphi(a)) pV^{i-1}(1) dV^j(1)\bigg) \\
&= V\bigg(\varphi(a)p^{i}  \cdot dV^j(1) \bigg) \\
&= a p^{i+1} \cdot dV^{j+1}(1) \\
&= px  dV^{j+1}(1) \\
&= x V(y).
\end{align*}

To prove $FdV = d$, we begin with a term $x = s_{\varphi}(a) V^i(1) \in E_n^0$ and compute 
\begin{align*}
FdV(x) &= Fd\bigg(s_{\varphi}\left(\varphi^{-1}(a)\right) V^{i+1}(1)\bigg) \\
&= F \bigg( \varphi^{-1}(a)  \cdot dV^{i+1}(1) \bigg) \\
&= a \cdot dV^i(1) \\
&= dx,
\end{align*}
as required.  

To complete the proof, it remains to prove Equation~(\ref{Teich relation}).  For fixed $n \geq 2$, we compute
\begin{align*}
Fd[a] &= Fd \left( \sum_{i = 0}^{n-1} s_{\varphi}(a_i) V^i(1) \right), \\
\intertext{where the $a_i$ are given by the formulas in Lemma~\ref{Teich lemma}.  We then compute further }  
&= \sum_{i = 1}^{n-1} F \bigg( a_i  \cdot dV^i(1) \bigg) \\
&= \sum_{i = 2}^{n-1}  \varphi(a_i)  \cdot dV^{i-1}(1) \\
&= \sum_{i = 1}^{n-2}  \varphi(a_{i+1}) \cdot dV^{i}(1) .
\intertext{For the other side of Equation~(\ref{Teich relation}), we have }
[a]^{p-1} d[a] &= \left( \sum_{j = 0}^{n-2} s_{\varphi}(a_j) V^j(1) \right)^{p-1} d\left(\sum_{i = 0}^{n-2} s_{\varphi}(a_i) V^i(1) \right) \\
&= \sum_{i = 1}^{n-2} \left( a_i \left( \sum_{j = 0}^{n-2} a_j p^j \right)^{p-1}  \right) \cdot dV^i(1).
\end{align*}  
We are finished if we can prove that, for every $i$ in the range $1 \leq i \leq n-2$, we have 
\[
\varphi(a_{i+1}) \equiv a_i \left( \sum_{j = 0}^{n-2} a_j p^j \right)^{p-1} \bmod p^iA,  
\]
which is clearly equivalent to proving 
\[
\varphi(a_{i+1}) \equiv a_i \left( \sum_{j = 0}^{i} a_j p^j \right)^{p-1} \bmod p^iA.  
\]
Because $\varphi$ is an isomorphism, it suffices to prove 
\[
\varphi^{i+1} (a_{i+1}) \equiv \varphi^{i}(a_i) \left( \sum_{j = 0}^{i} \varphi^i(a_j) p^j \right)^{p-1} \bmod p^iA.
\]
Recall our definition of the $a_j$ terms:
\[
[a] = \sum_{j = 0}^{\infty} s_{\varphi}(a_j) V^j(1) \in W(A).  
\]
Comparing the ghost components of the two sides, we have $a^{p^i} = \sum_{j = 0}^{i} \varphi^i(a_j) p^j$ for every $i \geq 0$.  Thus we are finished if we can prove 
\[
\varphi^{i+1} (a_{i+1}) \equiv \varphi^{i}(a_i) a^{p^i(p-1)} \bmod p^iA.
\]
By Lemma~\ref{Teich lemma}, we have reduced to showing 
\[
\frac{a^{p^{i+1}} - \left(\varphi(a)\right)^{p^{i}}}{p^{i+1}} \equiv \frac{a^{p^i} - \left(\varphi(a)\right)^{p^{i-1}}}{p^i} a^{p^i(p-1)} \bmod p^iA,
\]
which was proved in Lemma~\ref{congruence F lemma}.  This completes the proof of Equation~(\ref{Teich relation}), and this also completes the proof that $E_{\cdot}^{\bullet}$ is a Witt complex over $A$.  
\end{proof} 

\begin{corollary} \label{separated corollary}
For every integer $n$, the ring $E_{n}^{\bullet}$ is $p$-adically separated.  
\end{corollary} 

\begin{proof}
This follows immediately from our definition of $E_{n}^{\bullet}$:  in degree zero, $E_n^0 = W_n(A)$, which is $p$-adically separated because $A$ is $p$-adically separated.  In degree one, we have $p^{n-1} E_n^1 = 0$, and hence $E_n^1$ is also $p$-adically separated.  
\end{proof}

\begin{remark} \label{not iso} 
Our Witt complex $E_{\cdot}^{\bullet}$ is not isomorphic to the de\thinspace Rham-Witt complex $W_{\cdot} \Omega^{\bullet}_A$.  For example, $E_1^1 = 0$, while on the other hand it was shown in Corollary~\ref{de Rham non-zero} that $W_1\Omega^1_A = \Omega^1_A \neq 0$.  Nor is our Witt complex isomorphic to the relative de\thinspace Rham-Witt complex of Langer and Zink \cite{LZ04}: in their Witt complex, one always has $dV(1) = 0$.  Following the language of \cite[Remark~4.8]{Hes15}, our Witt complex $E_{\cdot}^{\bullet}$ is the $p$-typical de\thinspace Rham-Witt complex over $A$ relative to the $p$-typical $\lambda$-ring $(A, s_{\varphi})$: this follows from the fact that the elements $s_{\varphi}(\alpha)$ for $\alpha \in \Omega^1_A$ are all zero in $E_n^d$, and that the differential map $E_{\cdot}^{\bullet} \rightarrow E_{\cdot}^{\bullet+1}$ is $A$-linear.  
\end{remark}

\section{Applications to the de\thinspace Rham-Witt complex over $A = W(k)$} \label{app to W(k)}

Continue to assume $A = W(k)$ where $k$ is a perfect ring of odd characteristic~$p$.  In this section, we use our $p$-adically separated Witt complex $E_{\cdot}^{\bullet}$ from Section~\ref{E section} to give an explicit description (as an $A$-module) of the de\thinspace Rham-Witt complex over $A$.  

\begin{remark}
In this section we describe the de\thinspace Rham-Witt complex over $A = W(k)$ as an $A$-module.  The level~$n$ piece of the de\thinspace Rham-Witt complex over $A$ is always a $W_n(A)$-module.  We warn that the $W_n(A)$-module structure does not factor through restriction $W_n(A) \rightarrow W_1(A) \cong A$.  For example, multiplication by $V(1)$ is non-zero.
\end{remark}

As $W_{\cdot}\Omega^{\bullet}_A$ is by definition the initial object in the category of Witt complexes over $A$, we get a natural map  $W_{\cdot}\Omega^{\bullet}_A \rightarrow E_{\cdot}^{\bullet}$.  The following key result identifies the kernel of this map in degree~one.

\begin{proposition} \label{natural map}
Fix any integer $n \geq 1$, and let $S_n \subseteq W_n \Omega^1_A$ be the $W_n(A)$-submodule $\cap_{j = 1}^{\infty} p^j W_n \Omega^1_A$.  The natural map $\eta: W_n\Omega^1_A \rightarrow E_n^1$ induces an isomorphism $W_n\Omega^1_A/S_n \cong E_n^1$.  
\end{proposition} 

\begin{proof}
Because $E_n^1$ is $p$-adically separated, we see that $S_n$ is contained in the kernel of the map $W_n\Omega^1_A \rightarrow E_n^1$.  Consider the composition
\[
\Omega^1_{W_n(A)} \rightarrow W_n\Omega^1_A \rightarrow E^1_n.
\]
From our explicit description of $E^1_n$, we see that this composition is surjective.  We will now show that the kernel of this composition is generated as a $W_n(A)$-module by elements of the form
\begin{itemize}
\item $ds_{\varphi}(a)$,
\item $(V^j(1) - p^j) dV^i(1)$, and 
\item $p^i dV^i(1)$.
\end{itemize} 
It is clear that these groups of elements are all in the kernel.  

Consider now an arbitrary element $\omega \in \Omega^1_{W_n(A)}$ which is in the kernel; we must show that $\omega$ can be expressed as a $W_n(A)$-linear combination of the above elements.  Viewing $\Omega^1_{W_n(A)}$ as an $A$-module via $s_{\varphi}$, we have that an arbitrary element in $\Omega^1_{W_n(A)}$ can be expressed as an $A$-linear combination of the elements $V^i(1) ds_{\varphi}(a)$ and $V^j(1)dV^i(1)$ with $0 \leq j \leq i \leq n-1$.  Thus we may write
\[
\omega = \sum_{i = 0}^{n-1} s_{\varphi}(b_i) V^i(1) ds_{\varphi}(a_i) + \sum_{0 \leq j \leq i \leq n-1} s_{\varphi}(a_{j,i}) V^j(1) dV^i(1),
\]
for some elements $b_i, a_i, a_{j,i} \in A$.  Because the above itemized elements are all also in the kernel, we deduce that the element
\[
\omega' :=  \sum_{0 \leq j < i \leq n-1} p^j s_{\varphi}(a_{j,i})  dV^i(1)
\]
must also be in the kernel.  From the explicit description of $E_n^1$, because $\omega'$ is in the kernel of the composition, we have that for each fixed $i$, we have $\sum_j p^j a_{j,i} \in p^i A$.  Thus, for each fixed $i$, we have that $\sum_j p^j s_{\varphi}(a_{j,i})dV^i(1)$ is a $W_n(A)$-multiple of $p^i dV^i(1)$.  This proves that $\omega'$, and hence also $\omega$, is in the $W_n(A)$-submodule generated by the above elements.

We are finished, because $\Omega^1_{W_n(A)} \rightarrow W_n\Omega^1_A$ is surjective, and because the images of the above elements in $W_n \Omega^1_A$ are all in the submodule $S_n$.  In fact, the images of the second and third groups of elements are equal to 0 in $W_n\Omega^1_A$:  this follows from the identities $p^i dV^i = V^i d$ and 
\[
V(1)dV^i(1) = V\bigg(FdV^{i}(1)\bigg) = V\bigg(dV^{i-1}(1)\bigg) = pdV^i(1),
\]
which hold in every Witt complex.
\end{proof}

The following is modeled after \cite[Section~3.2]{HM03}.

\begin{lemma} \label{Mj module structure}
Continue to assume $A = W(k)$ where $k$ is a perfect ring of odd characteristic~$p$. 
For every $j \geq 1$, the map 
\[
h_j: A \rightarrow \Omega^1_A \oplus A, \qquad a \mapsto (-da, p^j a),
\]
is an $A$-module homomorphism, where the left-hand side has its $A$-module structure induced by $\varphi^j$ and where the right-hand side has component-wise addition and $A$-module multiplication defined by 
\[
x \cdot (\alpha, a) = \left(\varphi^j(x) \alpha - \frac{1}{p^j} a d\varphi^j(x) , \varphi^j(x) a\right).  
\]
\end{lemma}

\begin{remark}
For any element $z \in \Omega^1_A$, the term $\frac{1}{p^j} z$ makes sense in $\Omega^1_A$, because multiplication by $p$ is a bijection on $\Omega^1_A$.  
\end{remark}

\begin{proof}
We first check that the right-hand side is actually an $A$-module with respect to the structure we described.  It's clear that $(x_1 + x_2) \cdot (\alpha, a) = x_1 \cdot (\alpha,a) + x_2 \cdot (\alpha, a)$ and that $x \cdot \left( (\alpha_1, a_1) + (\alpha_2, a_2) \right) = x \cdot  (\alpha_1, a_1) + x \cdot (\alpha_2, a_2)$.  Next we compute
\begin{align*}
x_1 \cdot (x_2 \cdot (\alpha, a) )&= x_1 \cdot \left(\varphi^j(x_2) \alpha - \frac{1}{p^j} a d\varphi^j(x_2) , \varphi^j(x_2) a\right) \\
&= \left(\varphi^j(x_1) \left( \varphi^j(x_2) \alpha - \frac{1}{p^j} a d\varphi^j(x_2) \right) -  \frac{1}{p^j}   \varphi^j(x_2) ad\varphi^j(x_1), \varphi^j(x_1) \varphi^j(x_2) a\right) \\
&= \left(\varphi^j(x_1 x_2) \alpha - \frac{1}{p^j} a \varphi^j(x_1)d\varphi^j(x_2)  -  \frac{1}{p^j} a  \varphi^j(x_2) d\varphi^j(x_1), \varphi^j(x_1x_2) a\right) \\
&= (x_1 x_2) \cdot (\alpha, a).
\end{align*} 
Notice that so far the $1/p^j$ factor has played no role.  

Next we check that the proposed map is an $A$-module homomorphism; this is where the $1/p^j$ factor becomes important.  The map is clearly additive.  We then check that, on one hand, 
\begin{align*}
\varphi^j(x) a &\mapsto (-d(\varphi^j(x) a), p^j \varphi^j(x) a) \\
&= (-\varphi^j(x) d(a) - a d(\varphi^j(x)), p^j \varphi^j(x) a), \\
\intertext{and on the other hand,}
x \cdot (-da, p^j a) &= \left( -\varphi^j(x) da - \frac{1}{p^j} p^j a d(\varphi^j(x)), \varphi^j(x) p^j a \right).
\end{align*}
\end{proof} 

Let $M_j$ denote the cokernel of the $A$-module homomorphism $h_j$ from Lemma~\ref{Mj module structure}.  (This module is the analogue of what is denoted $_{h}W_{n}\omega^i_{(R,M)}$ in \cite[Section~3.2]{HM03}.)  We are going to describe the de\thinspace Rham-Witt complex over $A$ in terms of these modules $M_j$.  First we describe an $A$-module homomorphism $\Omega^1_A \rightarrow W_n \Omega^1_A$.  

Given any ring homomorphism $R \rightarrow S$, there is an induced $R$-module homomorphism $\Omega^1_R \rightarrow \Omega^1_S$.  In what follows, we will often use the following special case.  Let $s_{\varphi}: A \rightarrow W(A)$ be the ring homomorphism described in Proposition~\ref{sF map}.  For every $n \geq 1$, composing $s_{\varphi}$ with the restriction map induces a ring homomorphism $s_{\varphi}: A \rightarrow W_n(A)$ and hence an $A$-module homomorphism $s_{\varphi}: \Omega^1_A \rightarrow \Omega^1_{W_n(A)} \rightarrow W_n \Omega^1_A$.  If we want to be explicit about the codomain, we write $s_{\varphi,n}$ instead of $s_{\varphi}$.    

\begin{lemma} \label{sF degree one}
For every integer $n \geq 2$, the two $A$-module homomorphisms $s_{\varphi,n-1} \circ \varphi \circ \frac{1}{p}$ and $F \circ s_{\varphi,n}$ mapping $\Omega^1_A \rightarrow W_{n-1} \Omega^1_A$ are equal.   
\end{lemma}

\begin{proof}
It suffices to prove the images of a term $a_0 da_1$ are equal, and this follows from the relationships $dF = pFd$ and $s_{\varphi} \circ \varphi = F \circ s_{\varphi}$.
\end{proof}

\begin{lemma} \label{extend to morphism} 
Fix integers $n \geq j \geq 1$ and let $M_j$ be the cokernel of the $A$-module homomorphism $h_j$ from Lemma~\ref{Mj module structure}.  Consider $W_{n+1} \Omega^1_A$ as an $A$-module using the map $s_{\varphi}: A \rightarrow W(A)$.  The map
\begin{align*}
M_j &\rightarrow W_{n+1} \Omega^1_A, \\
(\alpha, a) &\mapsto V^j(s_{\varphi}(\alpha)) + dV^j(s_{\varphi}(a))
\end{align*}
is an $A$-module homomorphism.  
\end{lemma}

\begin{proof}
The map is clearly well-defined, because of the relation $p^jdV^j = V^jd$.  We have
\begin{align*}
x \cdot (\alpha, a) = (\varphi^j(x) &\alpha - \frac{1}{p^j} a d\varphi^j(x), \varphi^j(x) a)\\
 &\mapsto V^j \circ s_{\varphi}\bigg( \varphi^j(x) \alpha - \frac{1}{p^j} a d\varphi^j(x) \bigg) + dV^j \circ s_{\varphi} \bigg( \varphi^j(x) a \bigg) \\
&= V^j \bigg(F^j (s_{\varphi}(x)) s_{\varphi}(\alpha)\bigg) - V^j \bigg( \frac{1}{p^j} s_{\varphi}(a) dF^j(s_{\varphi}(x)) \bigg) + dV^j \bigg( F^j(s_{\varphi}(x)) s_{\varphi}(a) \bigg) \\
&= V^j \bigg(F^j (s_{\varphi}(x)) s_{\varphi}(\alpha)\bigg) - V^j \bigg( s_{\varphi}(a) F^jds_{\varphi}(x) \bigg) + d\bigg(s_{\varphi}(x) V^j(s_{\varphi}(a))\bigg)\\
&= s_{\varphi}(x) V^j ( s_{\varphi}(\alpha)) - V^j \left( s_{\varphi}(a) \right) ds_{\varphi}(x)  + V^j(s_{\varphi}(a))ds_{\varphi}(x) + s_{\varphi}(x) dV^j(s_{\varphi}(a))\\
&= s_{\varphi}(x) \bigg(V^j ( s_{\varphi}(\alpha))  +  dV^j(s_{\varphi}(a))\bigg).
\end{align*}  
\end{proof}

\begin{proposition} \label{hWn}
Continue to assume $A = W(k)$ where $k$ is a perfect ring of odd characteristic~$p$.  Fix any integer $n \geq 1$, and let $M_n$ be the cokernel of the $A$-module homomorphism from Lemma~\ref{Mj module structure}.  Consider $W_n \Omega^1_A$ and $W_{n+1}\Omega^1_A$ as $A$-modules via the ring homomorphism $s_{\varphi}: A \rightarrow W(A)$.  We have a short exact sequence of $A$-modules 
\begin{equation} \label{hW sequence}
0 \rightarrow M_n \rightarrow W_{n+1}\Omega^1_A \stackrel{R}{\longrightarrow} W_n\Omega^1_A \rightarrow 0,
\end{equation}
where the first map is given by 
\[
(\alpha, a)  \mapsto V^n(s_{\varphi}(\alpha)) + dV^n(s_{\varphi}(a)).
\]  
\end{proposition} 

\begin{proof}
Using Lemma~\ref{extend to morphism}, we see that these are maps of $A$-modules.  Then
using Proposition~\ref{Fil}, we reduce to proving that the map $M_n \rightarrow W_{n+1} \Omega^1_A$ is injective.  Assume $\alpha \in \Omega^1_A$ and $a \in A$ satisfy $V^n(s_{\varphi}(\alpha)) + dV^n(s_{\varphi}(a)) = 0 \in W_{n+1} \Omega^1_A$.  Then, because $\alpha$ is divisible by arbitrarily large powers of $p$, we have that $dV^n(s_{\varphi}(a))$ is divisible by arbitrarily large powers of $p$.  Write $a' = \varphi^{-n}(a)$.  We have
\begin{align*}
dV^n(s_{\varphi}(a)) &= d \left( s_{\varphi}(a') V^n(1) \right) \\
&= s_{\varphi}(a') dV^n(1) + V^n(1) ds_{\varphi}(a').
\end{align*}
The term $ds_{\varphi}(a')$ is divisible by arbitrarily large powers of $p$, so this implies $s_{\varphi}(a') dV^n(1)$ is divisible by arbitrarily large powers of $p$.  Thus by Corollary~\ref{separated corollary}, the image of $s_{\varphi}(a') dV^n(1)$ is equal to 0 in $E^1_{n+1}$, but then by our definition of $E^1_{n+1}$, we have that $a'$ is divisible by $p^n$, and hence so is $a = \varphi^n(a')$. 

Write $a = p^n a_0$.  We then have 
\[
0 = V^n(s_{\varphi}(\alpha)) + dV^n (s_{\varphi}(p^n a_0)) = V^n ( s_{\varphi}(\alpha + da_0)). 
\]
By Proposition~\ref{Omega bijection}, the map $p^n: \Omega^1_A \rightarrow \Omega^1_A$ is injective.  Because $p^n = F^n V^n$, we have that $V^n$ is also injective.   This shows that $\alpha = -da_0$, as claimed.
\end{proof}

\begin{remark}
Proposition~\ref{hWn} is the main result of this section.  The exactness claimed is mostly analogous to \cite[Proposition~3.2.6]{HM03}; the most interesting part of our result is the fact that the map $A \rightarrow \Omega^1_A \oplus A$ surjects onto the kernel of the map $\Omega^1_A \oplus A \rightarrow W_{n+1}\Omega^1_A$.  This result is difficult to prove because in general it is difficult to prove that elements in the de\thinspace Rham-Witt complex are non-zero.  See \cite[Proposition~2.2.1]{Hes05} for a result proving this same exactness in the context of the log de\thinspace Rham-Witt complex over the ring of integers in an algebraic closure of a local field.  See also \cite[Th\'{e}or\`{e}me~I.3.8]{Ill79} for a version of this result which is valid in characteristic~$p$.
\end{remark}

Using induction, we're able to give the following explicit description of $W_n\Omega^1_A$.  The key fact used by the construction is that the maps $\Omega^1_A \oplus A \rightarrow W_j \Omega^1_A$ given by $(\alpha, a) \mapsto V^{j-1}(\alpha) + dV^{j-1}(a)$ can be extended to maps into $W_n\Omega^1_A$ using $s_{\varphi}: A \rightarrow W(A)$.  

\begin{corollary} \label{de Rham-Witt as A-module}
Continue to assume $A = W(k)$ where $k$ is a perfect ring of odd characteristic~$p$.   View $W_{n+1}\Omega^1_A$ as an $A$-module using the ring homomorphism $s_{\varphi}: A \rightarrow W(A)$.  Let $M_0 = \Omega^1_A$, and for every $j \geq 1$, let $M_j = (\Omega^1_A \oplus A)/h_j(A)$ be the cokernel of the $A$-module homomorphism $h_j: a \mapsto (-da, p^j a)$ from Lemma~\ref{Mj module structure}.  For every integer $n \geq 2$, the map 
\begin{align*}
\prod_{j = 0}^n M_j &\rightarrow W_{n+1} \Omega^1_A \\
\intertext{induced by}
M_0 &\rightarrow W_{n+1} \Omega^1_A,\\
\alpha_0 &\mapsto s_{\varphi}(\alpha_0) \\
\intertext{and}
M_j &\rightarrow W_{n+1} \Omega^1_A \text{ for } j \geq 1, \\ 
(\alpha_j, a_j)  &\mapsto V^j(s_{\varphi}(\alpha_j)) + dV^j(s_{\varphi}(a_j))
\end{align*}
is an isomorphism of $A$-modules.  
\end{corollary} 

\begin{proof}
We know that the map is a homomorphism of $A$-modules by Lemma~\ref{extend to morphism}.  
For every integer $n \geq 1$, consider the complex
\[
\xymatrix{
0 \ar[r] & M_n \ar[d] \ar[r] &  \prod_{j = 0}^n M_j  \ar[d] \ar[r]   \ar[r] & \prod_{j = 0}^{n-1} M_j \ar[d] \ar[r] & 0\\
0 \ar[r] & M_n \ar[r] & W_{n+1}\Omega^1_A \ar[r] & W_{n} \Omega^1_A \ar[r]   & 0.\\
}
\]
The top row is clearly exact.  The bottom row is exact by Proposition~\ref{hW sequence}.  The right-hand vertical map is an isomorphism by induction.
Thus we're finished by the Five Lemma.
\end{proof} 

Similar, but easier, arguments work also for degrees $d \geq 2$.  Our applications involve degree $d = 1$, so we indicate the results more briefly.

\begin{proposition}
For every $d \geq 2$, $n \geq 1$, we have an exact sequence of $A$-modules
\[
0 \rightarrow \Omega^d_A \stackrel{V^n}{\longrightarrow} W_{n+1} \Omega^d_A \rightarrow W_n \Omega^d_A \rightarrow 0,
\]
where the $A$-module structure on $\Omega^d_A$ is given by $a \cdot \alpha := F^n(a)\alpha$, and where the $A$-module structure on the other two pieces is induced by $s_{\varphi}: A \rightarrow W(A)$.
\end{proposition}

\begin{proof}
The map $V^n: \Omega^d_A \rightarrow W_{n+1} \Omega^d_A$ is injective because $F^n \circ V^n = p^n$ is injective on $\Omega^d_A$.  We must also show that if $\omega \in W_{n+1}\Omega^d_A$ is in the kernel of $R$, then we can find $\alpha \in \Omega^d_A$ such that $\omega = V^n(\alpha)$.  We know that there exist $\alpha \in \Omega^d_A$ and $\beta \in \Omega^{d-1}_A$ such that 
\[
V^n(\alpha) + dV^n(\beta) = \omega.
\]
But now we're finished, because we can write $\beta = p^n \beta_0$ for some $\beta_0 \in \Omega^{d-1}_A$.  (This is where we use that $d \geq 2$.)
\end{proof}

We can deduce the following corollary in the same way as we deduced Corollary~\ref{de Rham-Witt as A-module}.

\begin{corollary} \label{deg geq 2}
For every $d \geq 2$ and every $n \geq 1$, we have an isomorphism of $A$-modules 
\[
\prod_{i = 0}^{n-1} \Omega^d_A \cong W_n \Omega^d_A,
\]
where the $A$-module structure on the $i$-th piece is given by $a \cdot \alpha_i := \varphi^i(a) \alpha_i.$
\end{corollary}

\begin{remark}
Much of the author's intuition for the de\thinspace Rham-Witt complex comes from the cases treated in Illusie's paper \cite{Ill79}, such as the description of the de\thinspace Rham-Witt complex over $\FF_p[t_1, \ldots, t_r]$ given in \cite[Section I.2]{Ill79}.   In this case, the de\thinspace Rham-Witt complex is 0 in degrees $d > r$.  We remark that the absolute, mixed characteristic de\thinspace Rham-Witt complex we are studying is very different.  Consider the easiest case of our setup, $A = \ZZ_p = W(\FF_p)$.  Then $\Omega^1_A$ is infinite-dimensional as a $\QQ_p$-vector space by Proposition~\ref{de Rham field}.  Thus $\Omega^d_A := \bigwedge^d \, \Omega^1_A$ is non-zero for all degrees $d$.  Thus in particular $W_n \Omega^d_A$ is non-zero for all integers $d \geq 0$ and $n \geq 1$.
\end{remark}

\begin{remark} \label{M Witt complex}
Corollary~\ref{de Rham-Witt as A-module} and Corollary~\ref{deg geq 2} give an explicit description of the $A$-module structure of the Witt complex $W_{\cdot}\Omega^{\bullet}_A$.  (Notice that for a general ring $B \neq W(k)$, we cannot expect a $B$-algebra structure on $W_{\cdot} \Omega^{\bullet}_B$.) It seems worthwhile to describe the entire Witt complex structure, at least for degrees $d = 0,1$, in terms of the description from Corollary~\ref{de Rham-Witt as A-module}.  Similar descriptions could be given for higher degrees.
\begin{itemize}
\item We already know the $A$-module structure, so to describe the $W_n(A)$-algebra structure on $W_n\Omega^1_A$, it suffices by Lemma~\ref{V-combination lemma} to describe the effect of multiplication by $V^j(1)$ on $\prod M_i$.  It sends all $M_i$ with $i \leq j$ into the $M_j$ component, via the formulas 
\begin{align*}
V^j(1) \cdot \alpha & = (\varphi^j (\alpha), 0)  \in M_j \text{ for } \alpha \in M_0 = \Omega^1_A, \text{ and } \\
V^j(1) \cdot (\alpha_i, a_i) &= (p^i \varphi^{j-i}(\alpha_i) + \varphi^{j-i}(da_i),0)  \in M_j,\\
&\hspace*{.5in} \text{for } (\alpha_i, a_i)  \in M_i, \text{ where } i \leq j.
\end{align*}
When $i \geq j$, multiplication by $V^j(1)$ acts on the $M_i$ component as multiplication by $p^j$.
\item To describe the differential $d: W_n(A) \rightarrow \prod M_i$, it suffices by Lemma~\ref{V-combination lemma} to note that $d: s_{\varphi}(a) \mapsto da \in M_0 = \Omega^1_A$ and that $d : s_{\varphi}(a_j) V^j(1)  \mapsto (0, \varphi^j(a_j))  \in M_j$ for $j \geq 1$.
\item The restriction map $R: \prod_{i = 0}^{n} M_i \rightarrow  \prod_{i=0}^{n-1} M_i$ is the obvious projection map.
\item To describe the map $V: \prod_{i = 0}^{n} M_i \rightarrow \prod_{i=0}^{n+1} M_i$, we note that 
\begin{align*}
V: \alpha &\mapsto (\alpha, 0)  \in M_1, \text{ where } \alpha \in M_0 = \Omega^1_A, \text{ and } \\
V: (\alpha_i, a_i) &\mapsto (\alpha_i, pa_i) \in M_{i+1}, \text{ where } (\alpha_i, a_i)  \in M_i.
\end{align*}
\item To describe the map $F: \prod_{i = 0}^{n+1} M_i \rightarrow \prod_{i=0}^{n} M_i$, we note that 
\begin{align*}
F: \alpha &\mapsto \varphi(\alpha) \in M_0, \text{ for } \alpha \in M_0 = \Omega^1_A \\
F: (\alpha_1, a_1) &\mapsto p\alpha_1 + da_1 \in M_0, \text{ for } (\alpha_1, a_1)  \in M_1, \text{ and } \\
F: (\alpha_i, a_i)  &\mapsto (p\alpha_i, a_i) \in M_{i-1}, \text{ for } (\alpha_i, a_i)  \in M_i.
\end{align*}
\end{itemize}
\end{remark}

\begin{corollary}
For every $n \geq 1$, the $p$-torsion submodule of $W_n \Omega^1_A$ is isomorphic to the free $A/p$-module of rank $n-1$ generated by $p^{j-1} dV^{j}(1)$, for $j = 1, \ldots, n-1$.  
\end{corollary} 

\begin{proof}
Using the fact that multiplication by $p$ is a bijection on $\Omega^1_A$, we see that the $p$-torsion module in $M_j = (\Omega^1_A \oplus A)/h_j(A)$ is a free $A/pA$-module of rank~1 generated by $(0,p^{j-1})$.  Then from Corollary~\ref{de Rham-Witt as A-module}, we see that these elements together generate the $p$-torsion submodule of $W_n \Omega^1_A$.  In the factor $M_j \cong (\Omega^1_A \oplus A)/h_j(A)$, a representative $(\alpha, a)$ has element $a$ uniquely determined modulo~$p^jA$.  This shows that we have a relation 
\[
 \sum dV^{j} (p^{j-1} \varphi^{j}(a)) = 0
\]
only if each $a \in pA$.  This shows that the proposed elements are free generators, which completes the proof.  
\end{proof}

\section{The de\thinspace Rham-Witt complex over $A/xA$} \label{A/xA section}

As usual, let $p$ denote an odd prime, let $k$ denote a perfect ring of characteristic~$p$, and let $A = W(k)$.  There are two natural ways to lift elements from $A$ to $W(A)$: the first is our ring homomorphism $s_{\varphi}$, and the second is the multiplicative Teichm\"uller map.  So far in this paper, we have made extensive use of the ring homomorphism $s_{\varphi}$.  In this section and the next, we make more frequent use of the Teichm\"uller map.  The reason is that we will be studying the kernel of the natural ring homomorphism $W(A) \rightarrow W(A/xA)$ for $x \in A$, and $[x]$ is in this kernel whereas $s_{\varphi}(x)$ in general is not.  For example, $[p]$ is in the kernel of $W(\ZZ_p) \rightarrow W(\ZZ_p/p\ZZ_p)$, whereas $s_{\varphi}(p) = p$ is not.

The exactness in Equation~(\ref{hW sequence}) above is very useful for making induction arguments involving the de\thinspace Rham-Witt complex.  For example, our proof of Corollary~\ref{de Rham-Witt as A-module} was dependent on our Witt complex $E_{\cdot}^{\bullet}$ only because $E_{\cdot}^{\bullet}$ was used to prove exactness in Equation~(\ref{hW sequence}). The goal of the remainder of the paper is to prove exactness of the corresponding sequence for the de\thinspace Rham-Witt complex over a certain class of perfectoid rings.  See \cite[Proposition~2.2.1]{Hes06} and \cite[Theorem~3.3.8]{HM03} for related results.  In future joint work with Irakli Patchkoria, we hope to use this exact sequence to provide algebraic proofs of results similar to Hesselholt's $p$-adic Tate module computation in \cite[Proposition~2.3.2]{Hes06}.  In this section we prove general results concerning $W_n \Omega^1_{(A/xA)}$ that are valid for arbitrary $x \in A$.  In Section~\ref{perfectoid section}, we specialize to a certain class of perfectoid rings, in which case we can prove stronger results, including the analogue of the exact sequence in Equation~(\ref{hW sequence}).

Fix an element $x \in A$.  For every integer $n \geq 1$, we have a surjective $A$-module homomorphism $W_n \Omega^1_A \rightarrow W_n \Omega^1_{(A/xA)}$, and Corollary \ref{de Rham-Witt as A-module} gives an explicit description of the domain.  We will give explicit $A$-module generators for the kernel.  Unfortunately, this kernel is not generated as an $A$-module by elements which are homogeneous with respect to the direct sum decomposition from Corollary~\ref{de Rham-Witt as A-module}. 

First we consider the case of level $n = 1$, which will be used repeatedly.

\begin{lemma} \label{n = 1 kernel}
The kernel of the $A$-module homomorphism $\Omega^1_A \rightarrow \Omega^1_{(A/xA)}$ is generated by $x \alpha$ for $\alpha \in \Omega^1_A$ together with the element $dx$.
\end{lemma}

\begin{proof}
This follows immediately from the usual right exact sequence  of $(A/xA)$-modules 
\begin{equation} \label{2nd sequence}
xA/x^2A \rightarrow \Omega^1_A \otimes_A (A/xA) \rightarrow \Omega^1_{(A/xA)} \rightarrow 0
\end{equation} 
\cite[Theorem~25.2]{Mat89}, where the left-most map is given by $xa  \mapsto d(xa) \otimes 1 $.
\end{proof}

Next we identify the kernel in the degree~zero case, $W_{\cdot} \Omega^0_A \rightarrow W_{\cdot} \Omega^0_{(A/xA)}$.  

\begin{lemma} \label{Witt quotient}
Let $K^0$ denote the kernel of the ring homomorphism $W(A) \rightarrow W(A/xA)$ induced by the projection $A \rightarrow A/xA$. Then $K^0$ consists precisely of elements of the form
\[
\sum_{k = 0}^{\infty} s_{\varphi}(a_k) V^k([x]),
\]
where $a_k \in A$.  
\end{lemma}

\begin{proof}
It's clear that these elements are in the kernel.  We now prove that an arbitrary element in the kernel can be written in this way.  Working one level at a time, it suffices to show that if $V^k(y_k)$ is in the kernel, then we can find $a_k \in A$ and $y_{k+1} \in W(A)$ such that 
\[
V^k(y_k) = s_{\varphi}(a_k) V^k([x]) + V^{k+1}(y_{k+1}).
\]
(Note that this also implies that $V^{k+1}(y_{k+1})$ is in the kernel.)  Because 
\[
s_{\varphi}(a_k) V^k([x]) = V^k(F^k(s_{\varphi}(a_k))[x]) = V^k(s_{\varphi}(\varphi^k(a_k)) [x])
\] 
and $\varphi: A \rightarrow A$ is surjective, we can find such elements $a_k$ and $y_{k+1}$.  
\end{proof}

We now do the same thing for the degree~one case.  In this case, the ring $W(A) \cong \varprojlim W_n(A)$ from Lemma~\ref{Witt quotient} gets replaced by the $W(A)$-module, $\varprojlim W_n \Omega^1_A$.   Corollary~\ref{de Rham-Witt as A-module} leads to an explicit description of this inverse limit as an $A$-module.  

More concretely, we give generators for the kernels of the $A$-module homomorphisms $W_n \Omega^1_A \rightarrow W_n \Omega^1_{(A/xA)}$, and we choose these generators so they are compatible under restriction maps for varying $n \geq 1$.  We view these generators as elements in $\varprojlim W_n \Omega^1_A$.  The main work involves studying, for particular choices of $A$ and $x$, the $A$-submodule of $\varprojlim W_n \Omega^1_A$ generated by these elements in the kernel.  Because these elements involve the Teichm\"uller lift $[x]$, they do not have a simple description in terms of our decomposition of $W_n \Omega^1_A$ given in Corollary~\ref{de Rham-Witt as A-module}.  

\begin{definition} \label{K definition} 
Let $M_0 = \Omega^1_A$ and for each integer $j \geq 1$, let $M_j$ be the cokernel of the $A$-module homomorphism in Lemma~\ref{Mj module structure}.  
Let $M$ denote the $A$-module 
\[
M = \prod_{j = 0}^{\infty} M_j.
\] 
Let $K^1 \subseteq M$ denote the $A$-submodule consisting of all elements of the form
\[
\sum_{k = 0}^{\infty} \left( V^k([x] s_{\varphi}(\alpha_k) ) + s_{\varphi}(a_k) dV^k([x]) \right),
\]
where $\alpha_k \in \Omega^1_A$ and where $a_k \in A$; here, to make sense of such an expression as an element in $M$, we use the structures described in Remark~\ref{M Witt complex}. 
\end{definition} 

\begin{remark}
\begin{enumerate}
\item By Corollary~\ref{de Rham-Witt as A-module}, $M$ is isomorphic as an $A$-module to $\varprojlim W_n \Omega^1_A$.
\item The $A$-module $K^1$ depends on our choice of element $x$, but that element is fixed throughout this section, so we write simply $K^1$ and not more suggestive notation such as $K^1_x$.
\end{enumerate}  
\end{remark}

We will use $K^1$ from Definition~\ref{K definition} to describe the kernel of $W_n \Omega^1_A \rightarrow W_n \Omega^1_{(A/xA)}$; namely, we will show that this kernel is the image of $K^1$ under the restriction map $R_n: M \rightarrow W_n \Omega^1_A$.

\begin{lemma} \label{kernel ideal}
For $n \geq 1$, write $R_n$ for the restriction map $W(A) \rightarrow W_n(A)$ and also for the restriction map $M \rightarrow W_n \Omega^1_A$.  The $A$-submodule of $W_n \Omega^{\bullet}_A$ generated by $R_n(K^0)$ and $R_n(K^1)$ and all higher degree terms ($W_n \Omega^d_A$ for $d \geq 2$) is an ideal in the ring $W_n \Omega^{\bullet}_A$.  
\end{lemma}

\begin{proof}
We have to show that the $A$-module generated by these elements is closed under multiplication by elements in $W_n \Omega^{\bullet}_A$.  
Consider an element $V^k([x]) m$, where $m \in W_n \Omega^1_A$.  This can be rewritten as $V^k([x] m_0)$, where $m_0 = F^k(m)$.  The element $m_0$ can be written (not uniquely) as 
\begin{align*}
m_0 &= s_{\varphi}(\alpha_0) + \sum_{i = 1}^{n-k-1} \left( V^i(s_{\varphi}(\alpha_i)) + dV^i(s_{\varphi}(a_i) )\right), \\
\intertext{and so}
[x]m_0 &= [x] s_{\varphi}(\alpha_0) + \sum_{i = 1}^{n-k-1} \left( [x]V^i( s_{\varphi}(\alpha_i)) + [x]dV^i(s_{\varphi}(a_i) )\right) \\
 &= [x] s_{\varphi}(\alpha_0) + \sum_{i = 1}^{n-k-1} \left( V^i( [x]^{p^i} s_{\varphi}(\alpha_i)) + dV^i([x]^{p^i} s_{\varphi}(a_i) ) - V^i(s_{\varphi}(a_i) [x]^{p^i - 1} d[x]\right).
\intertext{(Here we used the formula $Fd[x] = [x]^{p-1} d[x]$.)  And so} 
V^k([x]m_0) &\in R_n(K^1).
\end{align*}

Now we consider degree~1 terms in our $A$-module.  We first consider a term $V^k([x] s_{\varphi}(\alpha) )$ and then below we consider $dV^k([x])$.  We can write an arbitrary element $y \in W_n(A)$ as $\sum_{i = 0}^{n-1} s_{\varphi}(y_i) V^i(1)$, thus it suffices to show that  
\begin{align*}
V^k([x] s_{\varphi}(\alpha) ) V^i(1) &\in R_n(K^1).\\
\intertext{If $i \leq k$, we have }
V^k([x] s_{\varphi}(\alpha) ) V^i(1) &= V^k([x] s_{\varphi}(p^i  \alpha) ) \in R_n(K^1). \\
\intertext{If $i > k$, we have }
V^k([x] s_{\varphi}(\alpha) ) V^i(1) &= V^i(F^{i-k}\left( [x] s_{\varphi}(\alpha) \right) ) \\
&= V^i([x]^{p^{i-k}} s_{\varphi}(\frac{1}{p^{i-k}} \varphi^{i-k}(\alpha) ) ) \in R_n(K^1).
\end{align*}

Similarly, we find
\begin{align*}
dV^k([x]) V^i(1) &= V^i ( dV^{k-i}([x] ) = p^i dV^k([x]) \text{ for } i \leq k\\
\intertext{and }
dV^k([x]) V^i(1) &=V^i ( F^{i-k}d[x] ) = V^i ([x] m ) \text{ for } i > k \text{ and } m \in W_n \Omega^1_A.
\end{align*}
It was shown in the degree~zero portion of our proof that this latter element is in $R_n(K^1)$.  
\end{proof}

\begin{proposition} \label{Witt complex over B}
Define $G_{\cdot}^{\bullet}$ by
\begin{align*}
G_n^0 &:= W_n(A)/R_n(K^0)\\
G_n^1 &:= W_n \Omega^1_A / R_n(K^1)\\
G_n^d &:= 0 \text{ for } d \geq 2.
\end{align*}
Equipped with the structure maps inherited from $W_{\cdot}\Omega^{\bullet}_A$, this is a Witt complex over $A/xA$.
\end{proposition}  

\begin{proof}
The main thing to verify is that all of the necessary maps are well-defined.  All the various relations required of a Witt complex will then hold automatically since they hold in $W_n \Omega^{\bullet}_A$.  

The fact that $G_n^{\bullet}$ is a ring follows from Lemma~\ref{kernel ideal}.  Define $\lambda: W_n(A/xA) \rightarrow G_n^0$ to be the unique map such that the composition $W_n(A) \rightarrow W_n(A/xA) \rightarrow G_n^0$ is the projection map; this is possible by Lemma~\ref{Witt quotient}.  To define the differential $d: G_n^0 \rightarrow G_n^1$, we check that $d(s_{\varphi}(a) V^k([x])) \in R_n(K^1)$, which follows because
\[
d(s_{\varphi}(a) V^k([x])) = s_{\varphi}(a) dV^k([x]) + V^k([x] F^kds_{\varphi}(a)) = s_{\varphi}(a) dV^k([x]) + V^k([x] s_{\varphi}(\frac{1}{p^k}d\varphi^k(a))), 
\]
where the last equality holds by Lemma~\ref{sF degree one}.  Because $R \circ R_n = R_{n-1}$, it is clear that the restriction map $R$ is well-defined.  The fact that $V$ is well-defined follows from $VdV^k = pdV^{k+1}$ and the fact that $K^1$ is closed under multiplication by arbitrary elements in $W(A)$.  

To check that $F$ is well-defined on $G_n^1$, we need to show that $F(R_n(K^1)) \subseteq R_{n-1}(K^1)$, which means that we need to evaluate $F$ on elements 
\[
\sum_{k = 0}^{\infty} \left( V^k([x] s_{\varphi}(\alpha_k) ) + s_{\varphi}(a_k) dV^k([x]) \right).
\]
The result is immediate from the de\thinspace Rham-Witt relations, but we need to be careful to treat the $k = 0$ case separately from the $k > 0$ case.  We have
\begin{align*}
F\left( [x] s_{\varphi}(\alpha_0) \right) &= [x]^p s_{\varphi}( \frac{1}{p} \varphi(\alpha_0)) \text{ and }\\
F\left( s_{\varphi}(a_0) d[x] \right) &= s_{\varphi}(\varphi(a_0)) [x]^{p-1} d[x],
\end{align*}
and these elements are in $R_{n-1}(K^1)$ by Lemma~\ref{kernel ideal}.  For $k \geq 1$, we have
\[
F\left( V^k([x] s_{\varphi}(\alpha_k) \right) \text{ and } F\left( s_{\varphi}(a_k) dV^k([x]) \right) \in R_{n-1}(K^1),
\]
because $FV = p$ and $FdV = d$.
\end{proof}

\begin{proposition}
We have an isomorphism of $A$-modules
\[
W_n \Omega^1_A / R_n(K^1)\cong W_n \Omega^1_{(A/xA)}.
\]
\end{proposition} 

\begin{proof}
Viewing $W_n\Omega^1_{(A/xA)}$ as a Witt complex over $A$, we have a map of $W_n(A)$-modules $W_n \Omega^1_A \rightarrow W_n \Omega^1_{(A/xA)}$ which induces a map $f: (W_n \Omega^1_A)/R_n(K^1) \rightarrow W_n \Omega^1_{(A/xA)}$.  Similarly, $G_{\cdot}^{\bullet}$ is a Witt complex over $A/xA$ by Proposition~\ref{Witt complex over B}, so we have a map of $W_n(A/xA)$-modules $g: W_n \Omega^1_{(A/xA)} \rightarrow (W_n \Omega^1_A)/R_n(K^1)$.  We claim that the compositions $gf$ and $fg$ are both the identity map.   

Because the maps $f$ and $g$ arise from maps of Witt complexes, the two triangles in the following diagram commute.  
\[
\xymatrix{
W_n \Omega^1_A / R_n(K^1) \ar@/^1pc/[rr]^f  && W_n \Omega^1_{(A/xA)} \ar@/^1pc/[ll]^g \\
& \Omega^1_{W_n(A/xA)} \ar@{->>}[ul] \ar@{->>}[ur]
}
\]
Then, because the diagonal maps are both surjective, a diagram chase shows that $fg$ and $gf$ are both the identity map.   
\end{proof}

We conclude this section with a technical result about $K^1$ that will be used in the following section.  We include it in this section because it is valid in a more general context than what we consider in Section~\ref{perfectoid section}.

\begin{notation} \label{Pn notation}
For every integer $n \geq 1$, let $P_n$ denote the property
\begin{itemize}
\item $P_n:$  If $z \in K^1$ and $R_n(z) = 0$, then we can write
\[
z = \sum_{k = n}^{\infty} \left( V^k([x] s_{\varphi}(\alpha_k) ) + s_{\varphi}(a_k) dV^k([x]) \right).
\]
\end{itemize}
\end{notation}

\begin{proposition} \label{Pn proposition}
Assume that $x \not\in pA$.  If Property $P_1$ holds, then for every integer $n \geq 1$, the property $P_{n}$ also holds.
\end{proposition}

\begin{proof}
We prove this using induction on $n$.  Thus assume we know that property $P_{n-1}$ holds for some $n \geq 2$, and assume we have $z \in K^1$ such that $R_n(z) = 0$.  By our induction hypothesis, we can assume 
\[
z = \sum_{k = n-1}^{\infty} \left( V^k([x] s_{\varphi}(\alpha_k) ) + s_{\varphi}(a_k) dV^k([x]) \right).
\]
The terms for $k \geq n$ do not affect the conclusion, so we can in fact assume 
\begin{align*}
z &= V^{n-1}([x]s_{\varphi}(\alpha)) + s_{\varphi}(a)dV^{n-1}([x]) \\
&= V^{n-1}([x]s_{\varphi}(\alpha)) + dV^{n-1}([x]F^{n-1}(s_{\varphi}(a))) - V^{n-1}([x] F^{n-1}(ds_{\varphi}(a)) ) \\
&= V^{n-1}([x]s_{\varphi}(\alpha - \frac{1}{p^{n-1}} d (\varphi^{n-1}(a)) ) + dV^{n-1}([x]s_{\varphi}(\varphi^{n-1}(a))).
\intertext{Using Proposition~\ref{hWn}, because we are assuming $R_n(z) = 0$ and that $x$ is not divisible  by $p$, we have that $a$ must be divisible by $p^{n-1}$, and we find}
z &= V^{n-1}\bigg([x]s_{\varphi}(\alpha - \frac{1}{p^{n-1}} d (\varphi^{n-1}(a)) + d([x]s_{\varphi}(\varphi^{n-1}(a/p^{n-1}))) \bigg). 
\intertext{
The fact that $R_n(z) = 0$ implies that 
\[
[x]s_{\varphi}(\alpha - \frac{1}{p^{n-1}} d (\varphi^{n-1}(a)) + d([x]s_{\varphi}(\varphi^{n-1}(a/p^{n-1}))) 
\]
satisfies the assumption in Property $P_1$. Hence we have}
z &= V^{n-1} \bigg( \sum_{k = 1}^{\infty} \left( V^k([x] s_{\varphi}(\alpha_k) ) + s_{\varphi}(a_k) dV^k([x]) \right) \bigg) \\
&= \sum_{k = 1}^{\infty} \left( V^{k+n-1}([x] s_{\varphi}(\alpha_k) ) + s_{\varphi}(\varphi^{1-n}(a_k)) V^{n-1}(dV^k([x])) \right) \\
&= \sum_{k = 1}^{\infty} \left( V^{k+n-1}([x] s_{\varphi}(\alpha_k) ) + s_{\varphi}(\varphi^{1-n}(p^{n-1}a_k)) dV^{k+n-1}([x]) \right). 
\end{align*}
This completes the proof of Property $P_n$.
\end{proof}

\begin{lemma} \label{equivalent P_n}
An element $z \in M$ can be written in the form
\begin{align*}
z &= \sum_{k = n}^{\infty} \left( V^k([x] s_{\varphi}(\alpha_k) ) + s_{\varphi}(a_k) dV^k([x]) \right)
\intertext{if and only if}
z &\in V^n(K^1) + dV^n(K_0).
\end{align*}
In particular, Property~$P_n$ is equivalent to the following:
\begin{itemize}
\item If $z \in K^1$ and  $R_n(z) = 0$, then we have 
\[
z \in V^n(K^1) + dV^n(K_0).
\]
\end{itemize}
\end{lemma}

\begin{proof}
This follows from the same sorts of manipulations as in the above proofs.   The most difficult of these manipulations is showing that 
\[
s_{\varphi}(a_n) dV^n([x]) \in V^n(K^1) + dV^n(K_0).
\]
Using Lemma~\ref{sF degree one} and the Leibniz rule, one checks that 
\[
s_{\varphi}(a_n) dV^n([x]) = V^n\left([x] s_{\varphi} \left( \varphi^n  \left( \frac{-1}{p^n} d(a_n) \right) \right) \right)+ dV^n\bigg(s_{\varphi}(\varphi^n(a_n))[x]\bigg)\in V^n(K^1) + dV^n(K_0).\\
\]
\end{proof}

Similar manipulations show the following.

\begin{lemma} \label{Vn+dVn submodule}
For every integer $n \geq 1$, we have that $V^n(K^1) + dV^n(K_0) \subseteq M$ is a $W(A)$-submodule.
\end{lemma}

\begin{proof}
It's clear that the collection of elements of the form
\[
z = \sum_{k = n-1}^{\infty} \left( V^k([x] s_{\varphi}(\alpha_k) ) + s_{\varphi}(a_k) dV^k([x]) \right)
\]
forms an $A$-module, so we reduce to proving that $V^n(K^1) + dV^n(K_0)$ is closed under multiplication by $V^i(1)$, for $i \geq 1$.  Consider first the case $i \geq n$.  We have
\begin{align*}
V^i(1) V^n(K^1) = V^i (p^n F^{i-n}(K^1)) &\subseteq V^n(K^1)\\
V^i(1) dV^n(K^0) = V^i (F^{i-n}d(K^0)) &\subseteq V^n(K^1).
\intertext{Next we consider the case $i < n$.  We have}
V^i(1) V^n(K^1) = V^n(p^i K^1)  &\subseteq V^n(K^1)\\
V^i(1) dV^n(K^0) = d \left( V^i(1) V^n(K^0) \right) - V^n(K^0) dV^i(1) &\subseteq dV^n(K^0).
\end{align*}
\end{proof}

We cannot expect Property~$P_1$ to hold in general, as the following example shows.  In the next section we will prove that Property~$P_1$ (and hence Property~$P_n$ for every $n$) holds when $A/xA$ is a perfectoid ring satisfying Assumption~\ref{perfectoid assumption} below.

\begin{example} \label{not true for Zp}
Consider the ring $A = \ZZ_p$ and the element $x = p \in \ZZ_p$.  Clearly 
\[
 d[p] \in W_2 \Omega^1_{\ZZ_p}
\]
restricts to $dp = 0$ in $\Omega^1_{\ZZ_p}$.  On the other hand, because 
\[
[p] \equiv p + V(p^{p-1} - 1) \bmod V^2(W(\ZZ_p)),
\]
we have 
\[
d[p] =  - dV(1) \in W_2\Omega^1_{\ZZ_p}.
\]
The exactness of sequence~(\ref{hW sequence}) shows this element cannot be written as a $\ZZ_p$-linear combination of terms in $V(p\Omega^1_{\ZZ_p}) = V(\Omega^1_{\ZZ_p})$ and $dV(p) = Vd1 = 0$.  
\end{example}

\section{Applications to the de\thinspace Rham-Witt complex over perfectoid rings} \label{perfectoid section}

As usual, $p$ in this section denotes an odd prime.  The term \emph{perfectoid} was originally used in the context of algebras over a field, but we work with the more general notion of \emph{perfectoid ring} which has since been defined; see Definition~\ref{perfectoid def} below.   Examples of rings satisfying our definition of perfectoid include the $p$-adic completion of $\ZZ_p[\zeta_{p^{\infty}}]$, the $p$-adic completion of $\ZZ_p[p^{1/p^{\infty}}]$, and $\OCp$.

Throughout this section, we let $B$ denote a perfectoid ring satisfying Assumption~\ref{perfectoid assumption} below, and we let $A = W(B^{\flat})$, where 
\[ 
B^{\flat} := \varprojlim_{x \mapsto x^p} (B/pB)
\] 
is the tilt of $B$.  The ring $B^{\flat}$ is a perfect ring of characteristic~$p$.  Let $\theta: A = W(B^{\flat}) \rightarrow B$ denote the map $\theta_1$ from \cite[Section~3]{BMS16}. This is the ``usual" $\theta$ map from $p$-adic Hodge theory.  We will not need the definition of $\theta$; we will only need that it is surjective and its kernel is a principal ideal (by our definition of perfectoid).  Throughout this section, $x \in A$ denotes a fixed choice of generator for this principal ideal.  

We now explicitly state our definition of perfectoid, following \cite{BMS16}.  

\begin{definition}[{\cite[Definition~3.5]{BMS16}}] \label{perfectoid def}
A commutative ring $B$ is called \emph{perfectoid} if it is $\pi$-adically complete and separated for some element $\pi \in B$ such that $\pi^p$ divides $p$, the Frobenius map $B/pB \rightarrow B/pB$ is surjective, and the kernel of $\theta: W(B^{\flat}) \rightarrow B$ is principal.  
\end{definition}

\begin{assumption} \label{perfectoid assumption}
We further assume that our perfectoid ring $B$ is $p$-torsion free and that there exists a $p$-power torsion element $\omega \in \Omega^1_B$ such that the annihilator of $\omega$ is contained in $p^n B$ for some integer $n \geq 1$.
\end{assumption}

\begin{remark}
\begin{enumerate}
\item Assumption~\ref{perfectoid assumption} is satisfied, for example, if the perfectoid ring $B$ is contained in $\OCp$ and contains $\zeta_p$.  We do not know an elementary argument for this.  Fontaine in \cite[Th\'eor\`eme~$1^\prime$]{Fon81} gives an elementary argument to show that $d\zeta_p$ is non-zero in $\Omega^1_R$, where $R = \mathcal{O}_{\overline{\QQ_p}}$.  Bhargav Bhatt has shown us an argument involving the cotangent complex (which was used above in the proof of Proposition~\ref{Omega bijection}) to deduce that $d\zeta_p \in \Omega^1_{\OCp}$ is non-zero.  Once one knows that $d \zeta_p \neq 0$, an elementary argument shows that Assumption~\ref{perfectoid assumption} is satisfied.  We hope to consider the question, ``How restrictive is Assumption~\ref{perfectoid assumption}?", in later applications.  
\item Our proofs in this section work for any quotient $A/xA$ satisfying Assumption~\ref{perfectoid assumption}, but we do not know any interesting examples where $A/xA$ is not perfectoid.  In particular, see the next point.
\item We have been careful throughout this paper to work with $W(k)$ where $k$ is a perfect ring, instead of restricting our attention to the case where $k$ is a perfect field.  That generality is essential for Assumption~\ref{perfectoid assumption} to be reasonable, because when $k$ is a perfect field, the only $p$-torsion free quotient of $W(k)$ is the zero ring.
\end{enumerate}
\end{remark}

The entire goal of this section is to prove Proposition~\ref{B sequence proposition} below, which identifies the kernel of restriction $W_{n+1} \Omega^1_B \rightarrow W_n \Omega^1_B$ in terms of $B$ and $\Omega^1_B$.  Using a spectral sequence argument, our result will follow easily from Property~$P_n$ described in Notation~\ref{Pn notation}.  By Proposition~\ref{Pn proposition}, it will suffice to prove Property~$P_1$, which loosely says that if an element in $W_n \Omega^1_A$ is in both $\ker \theta$ and in the kernel of restriction $R_1$ to $\Omega^1_A$, then the element can be written as $V(\alpha) + dV(a)$, where both $\alpha$ and $a$ are in $\ker \theta$.  We now begin the proof that Property~$P_1$ holds.

We will apply the following lemma to our fixed $x \in A$ which generates $\ker \theta$, but it also holds for arbitrary $x \in A$.

\begin{lemma} \label{dividing by p in W(A)}
Choose $y \in W(A)$ such that $[x] = s_{\varphi}(x) + V(y)$.  Then we have 
\[
[x]^p = s_{\varphi}(\varphi(x)) + py
\]
\end{lemma}

\begin{proof}
Apply $F$ to both sides of $[x] = s_{\varphi}(x) + V(y)$.  
\end{proof}

Property~$P_1$ concerns elements which are both in the kernel of $W_n(\theta): W_{n} \Omega^1_A \rightarrow W_n \Omega^1_{(A/xA)}$ and also in the kernel of restriction $R_1: W_n \Omega^1_A \rightarrow \Omega^1_A$.  The following lemma considers the case of a particular element which is obviously in this intersection.  

\begin{lemma} \label{obvious zero}
We have $[x] ds_{\varphi}(x) - s_{\varphi}(x) d[x] \in V(K^1) + dV(K^0).$
\end{lemma}

\begin{proof}
We use the notation from Lemma~\ref{dividing by p in W(A)}.  We compute
\begin{align*}
[x] ds_{\varphi}(x) - s_{\varphi}(x) d[x] &= \bigg(s_{\varphi}(x) + V(y) \bigg) ds_{\varphi}(x) - s_{\varphi}(x) d \bigg(s_{\varphi}(x) + V(y) \bigg) \\
&= V(y) ds_{\varphi}(x) - s_{\varphi}(x) dV(y) \\
&= V\big(y Fds_{\varphi}(x)\big) - d \big( s_{\varphi}(x) V(y) ) + V(y) ds_{\varphi}(x) \\
&= V\big(2y Fds_{\varphi}(x) \big) - dV(y F(s_{\varphi}(x))) \\
&= V\big(2y Fds_{\varphi}(x) \big) - dV(y [x]^p - py^2). \\
\intertext{Because the term $dV(y[x]^p) \in dV(K^0)$, we reduce to showing the following element is in $V(K^1)$. }
V\big(2y Fds_{\varphi}(x) \big) + dV(py^2) &= V \big( 2y Fds_{\varphi}(x) + 2ydy \big) \\
&= V\big( 2y \left( Fds_{\varphi}(x) + dy \right) \big) \\
&= V\big( 2y \left( Fd \left( [x] - V(y) \right) + dy \right) \big) \\
&= V\big( 2y \left( [x]^{p-1} d[x] - dy + dy \right) \big) \in V(K^1).
\end{align*}
This completes the proof.  
\end{proof}

\begin{lemma} \label{x-torsion}
If $x \alpha_1 = 0 \in \Omega^1_A$, then $[x] s_{\varphi}(\alpha_1) \in V(K^1)$.
\end{lemma}

\begin{proof}
The key idea is that, because multiplication by $p$ is a bijection on $\Omega^1_A$, we also have that $\frac{x \alpha_1}{p^N} = 0 \in \Omega^1_A$ for every integer $N \geq 1$.  Applying Frobenius to both sides, we have $\frac{\varphi(x) \varphi(\alpha_1)}{p^N} = 0 \in \Omega^1_A$.  We will apply this observation in the case $N = 2$.  

Use the same notation as in Lemma~\ref{dividing by p in W(A)}.  We have
\begin{align*}
[x] s_{\varphi}(\alpha_1) &= s_{\varphi}(x) s_{\varphi}(\alpha_1) + V(y) s_{\varphi}(\alpha_1) \\
&= V(y) s_{\varphi}(\alpha_1) \\
&= V(y  F(s_{\varphi}(\alpha_1))) \\
&= V\left(y s_{\varphi}\left( \frac{\varphi(\alpha_1)}{p} \right) \right) \\
&= V\left(py s_{\varphi}\left( \frac{\varphi(\alpha_1)}{p^2} \right) \right) \\
&= V\left( \bigg( [x]^p - s_{\varphi}(\varphi (x)) \bigg) s_{\varphi}\left( \frac{\varphi(\alpha_1)}{p^2} \right) \right) \\
&= V\left(  [x]^p s_{\varphi}\left( \frac{\varphi(\alpha_1)}{p^2} \right) \right) \in V(K^1).
\end{align*}
\end{proof}

Using Assumption~\ref{perfectoid assumption}, we have a $p$-power torsion element $\omega \in \Omega^1_B$ with annihilator contained in $p^nB$ for some integer $n \geq 1$.  For every integer $r \geq 1$, the following lemma enables us to produce a $p$-power torsion element $\eta \in \Omega^1_B$ with annihilator contained in $p^{n+r}B$.  

\begin{lemma} \label{dividing annihilator}
Assume $\omega \in \Omega^1_B$ is such that $\Ann \omega \subseteq p^n B$, where $n \geq 1$ is an integer.  If $\eta \in \Omega^1_B$ is an element such that $p^r \eta = \omega$ for some integer $r \geq 1$, then $\Ann \eta \subseteq p^{n+r}B$.  
\end{lemma}

\begin{proof}
It suffices to prove this in the case $r = 1$, so let $\eta \in \Omega^1_B$ be such that $p \eta = \omega$.  Let $b \in \Ann \eta$.  Then in particular $b \in \Ann \omega$, so we can write $b = p^n b_0$ for some $b_0 \in B$.  Then we know
\[
0 = b \eta = p^n b_0 \eta = p^{n-1} b_0 \omega,
\]
and hence $p^{n-1} b_0 \in p^n B$.  Assumption~\ref{perfectoid assumption} requires that $B$ is $p$-torsion free, so we deduce that $b_0 \in pB$, and hence $b \in p^{n+1}B$, as required.
\end{proof}

The following is the most important of the preliminary results in this section.  If we could prove Proposition~\ref{a in xA} without using the element $\omega$ from Assumption~\ref{perfectoid assumption}, then the results of this section would hold for all $p$-torsion free perfectoid rings. 

\begin{proposition} \label{a in xA}
If $a dx \in x \Omega^1_A$, then $a \in xA$.  
\end{proposition}

\begin{proof}
Our hypothesis implies $a \frac{dx}{p^N} \in \ker \theta$ for every integer $N \geq 0$, and we will show this implies $\displaystyle \theta(a) \in \cap \, p^r B = 0$.

Fix an integer $N \geq 1$.  Because $\theta: A \rightarrow B$ is surjective, we know the induced map $\Omega^1_{A} \rightarrow \Omega^1_B$ is surjective.  Let $\omega_A \in \Omega^1_A$ map to the element $\omega \in \Omega^1_B$ described in Assumption~\ref{perfectoid assumption}.  Because $\omega$ is $p$-power torsion, we know that $p^m \omega_A \in x \Omega^1_A + A dx$ for some integer $m \geq 1$.  Thus, for every integer $N \geq 1$, we can write $\frac{1}{p^{N-m}} \omega_A = x \alpha_N + a_N \frac{dx}{p^{N}}$ for some $\alpha_N \in \Omega^1_A$ and $a_N \in A$.  

Consider now the element $a dx \in x\Omega^1_A$ from the statement of this proposition.  We deduce that $a \frac{dx}{p^N} \in x\Omega^1_A$ for every integer $N \geq 1$, so $a \frac{dx}{p^N} \in \ker \theta$ for every integer $N \geq 1$.  If we multiply by the element $a_N$ from the previous paragraph, we know that $a a_N \frac{dx}{p^N}$ is in $\ker \theta$ for every integer $N \geq 1$.  If we apply $\theta$ to $a a_N \frac{dx}{p^N}$, we see that $\theta(a) \in B$ is in the annihilator of some element $\eta$ satisfying $p^{N-m} \eta = \omega$.  Thus, by Lemma~\ref{dividing annihilator}, we have that $\theta(a) \in B$ is divisible by arbitrarily large powers of $p$.  Thus $a \in xA$, as required.
\end{proof}

\begin{remark}
Proposition~\ref{a in xA} implies that for our particular rings $A$ and $A/xA$, the left-most map in the exact sequence (\ref{2nd sequence}) is injective.
\end{remark}

\begin{proposition} \label{lift to kernel}
If $x \alpha + a dx = 0 \in \Omega^1_A$, then $[x] s_{\varphi}(\alpha) + s_{\varphi}(a) d[x] \in V(K^1) + dV(K^0)$.  
\end{proposition} 

\begin{proof}
We have $adx = - x \alpha$, so by Proposition~\ref{a in xA}, we know that $a = xa_1$ for some $a_1 \in A$, and thus our assumption means $x (\alpha + a_1 dx) = 0 \in \Omega^1_A$.  By Lemma~\ref{x-torsion}, we know that $[x] (s_{\varphi}(\alpha) + s_{\varphi}(a_1) d(s_{\varphi}(x)) ) \in V(K^1).$  Thus it suffices to show that 
\[
[x]s_{\varphi}(a_1) ds_{\varphi}(x) - s_{\varphi}(x) s_{\varphi}(a_1) d[x] \in V(K^1) + dV(K^0).
\]
Thus, by Lemma~\ref{Vn+dVn submodule}, it suffices to show that 
\[
[x] ds_{\varphi}(x) - s_{\varphi}(x) d[x] \in V(K^1) + dV(K^0).
\]
So we're done by Lemma~\ref{obvious zero}.  
\end{proof} 

Consider now an arbitrary element $y \in K^1$, 
\[
y = \sum_{k = 0}^{\infty} \left( V^k([x] s_{\varphi}(\alpha_k) ) + s_{\varphi}(a_k) dV^k([x]) \right),
\]
and assume it restricts to 0 in level one, i.e., assume $R_1(y) = 0 \in \Omega^1_A$.  This means that 
\[
x \alpha_0 + a_0 dx = 0 \in \Omega^1_A.  
\]
Then Proposition~\ref{lift to kernel} shows that Property~$P_1$ from Notation~\ref{Pn notation} holds.  We immediately deduce the following from Proposition~\ref{Pn proposition}.   

\begin{corollary} \label{Pn holds}
For every $n \geq 1$, Property~$P_n$ from Notation~\ref{Pn notation} holds.
\end{corollary}

The following result is the main result of this section.  It is modeled after \cite[Proposition~3.2.6]{HM03}.  Compare also Proposition~\ref{hWn}.

\begin{proposition} \label{B sequence proposition}
Let $B$ be a perfectoid ring satisfying Assumption~\ref{perfectoid assumption}.   
For every integer $n \geq 1$, we have a short exact sequence of $W_{n+1}(B)$-modules 
\begin{equation} \label{B sequence}
0 \rightarrow B \rightarrow \Omega^1_B \oplus B \rightarrow W_{n+1}\Omega^1_B \stackrel{R}{\rightarrow} W_n \Omega^1_B \rightarrow 0,
\end{equation}
where the maps and $W_{n+1}(B)$-module structure are defined as follows.  The map $B \rightarrow \Omega^1_B \oplus B$ is given by $b \mapsto (-db, p^n b)$.  The map $\Omega^1_B \oplus B \rightarrow W_{n+1} \Omega^1_B$ is given by $(\beta, b) \mapsto V^n(\beta) + dV^n(b)$.  The $W_{n+1}(B)$-module structure on $B$ is given by $F^n$.  The $W_{n+1}(B)$-module structure on $\Omega^1_B \oplus B$ is given by
\[
y \cdot (\omega, b) = \big(F^n(y) \omega - b F^n\left( dy \right), F^n(y) b\big) ,\text{ where } y \in W_{n+1}(B).
\] 
The $W_{n+1} (B)$-module structure on $W_n \Omega^1_B$ is induced by restriction.
\end{proposition}

\begin{proof}
Consider the following short exact sequence of chain complexes (the chain complexes are written horizontally, and the short exact sequences are written vertically):
\[
\xymatrix{
0 \ar[r] & 0 \ar[r] & 0 \ar[r] & 0 \ar[r] & 0 \ar[r] & 0 \\
0 \ar[r] & B \ar[r] \ar[u] & \Omega^1_B \oplus B \ar[r]\ar[u] & W_{n+1} \Omega^1_B \ar[r]\ar[u] & W_n \Omega^1_B \ar[r] \ar[u]& 0 \\
0 \ar[r]  & A \ar[r] \ar[u]^{\theta} & \Omega^1_A \oplus A \ar[r] \ar[u]^{\theta} & W_{n+1} \Omega^1_A \ar[u]^{W_{n+1}(\theta)} \ar[r] & W_n \Omega^1_A \ar[u]^{W_{n}(\theta)} \ar[r] & 0\\
0 \ar[r] & xA \ar[r] \ar[u]  & R_1(K^1) \oplus R_1(K^0) \ar[r] \ar[u]  & R_{n+1}(K^1)  \ar[r] \ar[u] & R_n(K^1)  \ar[r] \ar[u] & 0\\
0 \ar[r] & 0 \ar[r]\ar[u] & 0 \ar[r]\ar[u] & 0 \ar[r] \ar[u]& 0 \ar[r] \ar[u]& 0.
}
\]
For convenience, write these chain complexes as $0 \rightarrow K_{\bullet} \rightarrow A_{\bullet} \rightarrow B_{\bullet} \rightarrow 0$, where we consider the complexes concentrated in degrees 0 to 3.  We must show that $H_n(B_{\bullet}) \cong 0$ for all $n$.  It's trivial that $H_0(B_{\bullet}) \cong 0$ and $H_3(B_{\bullet}) \cong 0$.  Using Proposition~\ref{Fil}, we have also that $H_1(B_{\bullet}) \cong 0$.  This leaves $H_2(B_{\bullet})$.

Consider now the long exact sequence in homology \cite[Theorem~1.3.1]{Wei94} associated to the above short exact sequence of chain complexes.  By Proposition~\ref{hWn}, we have that $H_n(A_{\bullet}) \cong 0$ for all $n$.   It follows that $H_2(B_{\bullet}) \cong H_1(K_{\bullet})$.  We will finish the proof by showing that $H_1(K_{\bullet}) \cong 0$.

Consider an element in $R_{n+1}(K^1)$ which restricts to 0 in $W_n\Omega^1_A$.  By Corollary~\ref{Pn holds}, we know that this element can be written as $V^n([x] s_{\varphi}(\alpha_n)) + s_{\varphi}(a_n) dV^n([x])$, for some $\alpha_n \in \Omega^1_A$ and some $a_n \in A$.  By Lemma~\ref{equivalent P_n}, such an element lies in $V^n(K^1) + dV^n(K^0)$, and hence is in the image of the map 
\[
R_1(K^1) \oplus R_1(K^0) \stackrel{V^n + dV^n}{\longrightarrow}  R_{n+1}(K^1).
\]
This shows that $H_1(K_{\bullet}) \cong 0$, and hence that $H_2(B_{\bullet}) \cong 0$, as required.
\end{proof}

\begin{example}
As in Example~\ref{not true for Zp}, the analogue of exactness in Equation~(\ref{B sequence}) does not hold for arbitrary quotients of a ring $A = W(k)$.  For example, exactness does not hold for $B = \ZZ_p/p\ZZ_p \cong \ZZ/p\ZZ$.  In this case, not even the left-most map $B \rightarrow \Omega^1_B \oplus B$ is injective.  More significantly, we know $W_{n+1} \Omega^1_{(\ZZ/p\ZZ)}$ is zero for all $n$, so $dV^n(1) = 0 \in W_{n+1} \Omega^1_{(\ZZ/p\ZZ)}$ for all $n \geq 1$.  By contrast, Proposition~\ref{B sequence proposition} shows that $dV^n(1) \neq 0$ for all perfectoid rings $B$ satisfying Assumption~\ref{perfectoid assumption}.
\end{example}

\begin{remark}
Assume $B$ is a ring for which the sequence in Equation~(\ref{B sequence}) is exact.  Assume $B_0 \subseteq B$ is a subring satisfying the following two properties:
\begin{enumerate}
\item We have $p^n B \cap B_0 = p^n B_0$.
\item The $B_0$-module homomorphism $\Omega^1_{B_0} \rightarrow \Omega^1_{B}$ is injective.
\end{enumerate}
It then follows that the analogue of Equation~(\ref{B sequence}) for $B_0$ is also exact.  In foreseeable applications, verifying the first condition will be trivial, but in general it may be difficult to verify the second condition.  For example, if $B$ is $\OCp$ and $B_0$ is the valuation ring in an algebraic extension of $\QQ_p$, it is not clear whether we should expect the second condition to hold.  For this reason, this remark might be more useful in the context of Hesselholt's \cite[Proposition~2.2.1]{Hes06}, which shows exactness of a log analogue of Equation~(\ref{B sequence}) when $B = \mathcal{O}_{\overline{\QQ_p}}$.  
\end{remark}

\begin{remark}
In this section and the previous section, we have been working with an explicit quotient of the de\thinspace Rham-Witt complex over $A = W(k)$.  Perhaps similar results could be attained by working with an explicit quotient of the de\thinspace Rham-Witt complex over the polynomial algebra $A[t]$.  An explicit description of the de\thinspace Rham-Witt complex over $A[t]$ is given, in terms of the de\thinspace Rham-Witt complex over $A$, in \cite[Theorem~B]{HM04}.
\end{remark}

\bibliography{Tate}

\begin{thebibliography}{10}

\bibitem{BMS16}
Bhargav Bhatt, Matthew Morrow, and Peter Scholze.
\newblock Integral p-adic hodge theory.
\newblock {\em arXiv preprint arXiv:1602.03148}, 2016.

\bibitem{Cos08}
Viorel Costeanu.
\newblock On the 2-typical de {R}ham-{W}itt complex.
\newblock {\em Doc. Math.}, 13:413--452, 2008.

\bibitem{Fon81}
Jean-Marc Fontaine.
\newblock Formes diff\'erentielles et modules de {T}ate des vari\'et\'es
  ab\'eliennes sur les corps locaux.
\newblock {\em Invent. Math.}, 65(3):379--409, 1981/82.

\bibitem{Hes05}
Lars Hesselholt.
\newblock The absolute and relative de {R}ham-{W}itt complexes.
\newblock {\em Compos. Math.}, 141(5):1109--1127, 2005.

\bibitem{Hes06}
Lars Hesselholt.
\newblock On the topological cyclic homology of the algebraic closure of a
  local field.
\newblock In {\em An alpine anthology of homotopy theory}, volume 399 of {\em
  Contemp. Math.}, pages 133--162. Amer. Math. Soc., Providence, RI, 2006.

\bibitem{Hes15}
Lars Hesselholt.
\newblock The big de {R}ham-{W}itt complex.
\newblock {\em Acta Math.}, 214(1):135--207, 2015.

\bibitem{HM03}
Lars Hesselholt and Ib~Madsen.
\newblock On the {$K$}-theory of local fields.
\newblock {\em Ann. of Math. (2)}, 158(1):1--113, 2003.

\bibitem{HM04}
Lars Hesselholt and Ib~Madsen.
\newblock On the {D}e {R}ham-{W}itt complex in mixed characteristic.
\newblock {\em Ann. Sci. \'Ecole Norm. Sup. (4)}, 37(1):1--43, 2004.

\bibitem{Ill79}
Luc Illusie.
\newblock Complexe de de\thinspace {R}ham-{W}itt et cohomologie cristalline.
\newblock {\em Ann. Sci. \'Ecole Norm. Sup. (4)}, 12(4):501--661, 1979.

\bibitem{LZ04}
Andreas Langer and Thomas Zink.
\newblock De {R}ham-{W}itt cohomology for a proper and smooth morphism.
\newblock {\em J. Inst. Math. Jussieu}, 3(2):231--314, 2004.

\bibitem{Mat89}
Hideyuki Matsumura.
\newblock {\em Commutative ring theory}, volume~8 of {\em Cambridge Studies in
  Advanced Mathematics}.
\newblock Cambridge University Press, Cambridge, second edition, 1989.
\newblock Translated from the Japanese by M. Reid.

\bibitem{stacks-project}
The {Stacks Project Authors}.
\newblock {S}tacks {P}roject.
\newblock \url{http://stacks.math.columbia.edu}, 2017.

\bibitem{Wei94}
Charles~A. Weibel.
\newblock {\em An introduction to homological algebra}, volume~38 of {\em
  Cambridge Studies in Advanced Mathematics}.
\newblock Cambridge University Press, Cambridge, 1994.

\end{thebibliography}
\bibliographystyle{plain}

\end{document}